\documentclass[12pt,reqno]{amsart}
\usepackage{graphicx,xcolor,pdfsync,fullpage}
\usepackage{amssymb,enumerate}
\usepackage[bookmarksnumbered,colorlinks,plainpages,backref]{hyperref}
\usepackage{ulem}

\usepackage{amsmath, amsthm,mathrsfs,enumitem}

\newtheorem{thm}{Theorem}[section]
\newtheorem{cor}[thm]{Corollary}
\newtheorem{lem}[thm]{Lemma}
\newtheorem{prop}[thm]{Proposition}
\theoremstyle{definition}
\newtheorem{defn}[thm]{Definition}

\theoremstyle{remark}
\newtheorem{rem}[thm]{Remark}
\numberwithin{equation}{section}
\newcommand{\norm}[1]{\left\Vert#1\right\Vert}
\newcommand{\abs}[1]{\left\vert#1\right\vert}
\newcommand{\set}[1]{\left\{#1\right\}}

\newcommand{\dbar}{\bar\partial}
\newcommand{\ddbar}{\partial\bar\partial}

\newcommand{\tmop}[1]{\ensuremath{\operatorname{#1}}}
\renewcommand{\Re}{\tmop{Re}}

\DeclareMathOperator{\range}{Range}
\DeclareMathOperator{\re}{Re}
\DeclareMathOperator{\im}{Im}

\newcommand{\rl}{{\mathbb{R}}}

 \DeclareMathOperator{\tr}{Tr}

\newcommand{\red}[1]{\textcolor{red}{#1}}

\newcommand{\violet}[1]{\textcolor{violet}{#1}}

\DeclareMathOperator{\Tr}{Tr}

\newcommand{\Om}{\Omega}
\newcommand{\opI}{\mathcal{I}}

\newcommand{\qt}{\widetilde{q}}

\newcommand{\dbars}{\bar\partial^*}

\DeclareMathOperator{\Dom}{Dom}

\newcommand{\bd}{\partial}

\newcommand{\p}{\partial}

\newcommand{\R}{\mathbb R}

\newcommand{\N}{\mathbb N}
\newcommand{\C}{\mathbb C}
\newcommand{\opL}{\mathcal{L}}

\newcommand{\cx}{\mathbb{C}}
\newcommand{\ol}{\overline}
\newcommand{\nn}{\nonumber}
\newcommand{\ep}{\epsilon}
\newcommand{\I}{\mathcal{I}}

\newcommand{\vp}{\varphi}

\newcommand{\om}{\omega}
\newcommand{\omb}{\bar\omega}
\newcommand{\dbarsvp}{\bar\partial^*_\varphi}

\newcommand{\ipr}[1]{\left\langle #1 \right\rangle}

\begin{document}

\title{The $\bar\partial$-problem on $Z(q)$-domains}%

\author{Debraj Chakrabarti}
\address{Department of Mathematics,
	Central Michigan University,
	Mt. Pleasant, MI 48859,
	USA}
\email{chakr2d@cmich.edu}
\author{Phillip S. Harrington}
\author{Andrew Raich}%
\address{Department of Mathematical Sciences, SCEN 309, 1 University
		of Arkansas, Fayetteville, AR 72701, USA}
\email{psharrin@uark.edu,\ araich@uark.edu}
\begin{abstract}
\end{abstract}
	
\thanks{The first author was partially supported by a  US National Science
	Foundation grant number DMS-2153907, and by a grant from the Simons
	Foundation (706445, DC).
The third author was partially supported by a grant from the Simons Foundation (707123, ASR)}

\begin{abstract} Given a complex manifold containing a relatively compact $Z(q)$ domain, we give sufficient geometric conditions on the domain so that its $L^2$-cohomology in degree $(p,q)$ (known to be finite dimensional) vanishes.  The condition consists of	 the existence of a  smooth weight function in a neighborhood of the closure of the domain, where the complex Hessian of the weight has a prescribed number of eigenvalues of a particular sign, along with good interaction at the boundary of the Levi form with the complex Hessian, encoded in a subbundle of common positive directions for the two Hermitian forms.
\end{abstract}


\maketitle

%
%

\section{Introduction}
\subsection{Solvability of the $\dbar$-problem}
In this paper, we give a sufficient geometric condition for a relatively compact $Z(q)$ domain $\Om$ in a complex manifold to have vanishing Dolbeault cohomology at level $(p,q)$. In particular, our condition requires
a smooth weight function defined in a neighborhood of $\bar\Om$ whose complex Hessian has a given number of positive and negative eigenvalues. Our technique is to build a metric that turns a condition about \emph{numbers} of
eigenvalues into one about \emph{sums} of eigenvalues. Typically, the former conditions are invariant under biholomorphisms while $L^2$ methods require the latter conditions. The heart of our argument is an investigation of
an invariant condition in which the Levi form and the complex Hessian of the weight share positive directions.  This is novel and represents a new approach to rectify the dichotomy between invariant conditions and sufficient conditions to use $L^2$ methods.

 The solvability of the $\dbar$-problem on a domain in $\cx^n$ or on a complex manifold depends on certain convexity conditions, the most natural of which is being Stein, i.e.,
	the existence of a strictly plurisubharmonic exhaustion function.  Under this condition we get the vanishing of
	the Dolbeault cohomology in degree $(p,q)$ for $q\geq 1$, a special case of H.~Cartan's celebrated \emph{Theorem B} (see \cite{GuRo1965}). It was realized early that one can generalize this substantially: by well-known results of
	Andreotti and Grauert (see \cite{AnGr62}), if  there is a smooth exhaustion function  on the $n$-dimensional complex  manifold $M$  whose
	complex Hessian has $n-q+1$ positive eigenvalues, then $H^{p,q}(M)=0$ for each $p$. A smooth function on an $n$-dimensional complex manifold whose complex Hessian has $n-q+1$ positive eigenvalues at each point  is usually called a \emph{(strictly) $q$-convex function}, but there are other competing conventions for this and other related definitions, so we mostly avoid the use of the ``$q$-terminology" in this paper.   As is common in differential geometry,
	the convexity condition is encoded in a Hermitian form (cf.\ the second fundamental form in the classical theory of surfaces in Euclidean space).
		For more on the $q$-convexity conditions, see \cite{ohsawaqconvex,EaSu80}.

	  In the study of complex function theory on a domain (open connected subset) $\Omega$ in a complex manifold $M$,
	 methods based on $L^2$-estimates are usually easier to use than the classical sheaf-theoretic arguments (see \cite{ChSh01, Dem12b}). Establishing $L^2$ estimates requires the choice of a Hermitian metric on the manifold.
	  The complex convexity of the boundary $\p\Om$
is encoded in its Levi form. The ``interior complex convexity" of the domain is encoded in a
 smooth ``weight function" $\varphi$ in a neighborhood of the closure $\ol{\Om}$, where the complex Hessian  of $\vp$
has certain positivity conditions imposed. One can interpret $\vp$ as giving rise to a Hermitian metric $e^{-\vp}$  on the trivial line bundle in a neighborhood of $\ol{\Omega}$.  By a well-known result of H\"ormander (see \cite{Hor65}), if the domain $\Omega$ is strictly pseudoconvex, and the weight $\vp$ is strictly plurisubharmonic,  then we have
 the vanishing of the  $L^2$-cohomology $H^{p,q}_{L^2}(\Omega)$ for $q\geq 1$.
  This is of course equivalent to the solvability, with estimates in the $L^2$-norm, of the $\dbar$-equation $\dbar u =g$  for a $\dbar$-closed square-integrable $(p,q)$-form $g$.

In analogy with convexity conditions on manifolds given by exhaustion or weight functions with complex Hessians of specified signature,	 one can also consider partial convexity conditions for the boundary 
in order to study the $\dbar$-problem in a fixed degree.  For  $1\leq q \leq n-1$, a smoothly bounded domain $\Omega$ in an $n$-dimensional  complex manifold
	is said to satisfy  \emph{condition $Z(q)$},
	if at each point of the boundary $\p\Omega$ the
	Levi form has at least $n-q$ positive eigenvalues, or has at least $q+1$ negative eigenvalues. This condition is fundamental in the theory of the $\dbar$-Neumann problem, since it is necessary and sufficient for $\frac{1}{2}$-subelliptic estimates on $(p,q)$-forms (see \cite{FoKo72}).
	
	 	  Another type of partial convexity condition arises from a consideration of the right-hand side of the  Bochner-Kohn-Morrey-H\"ormander identity for higher degree forms, and the condition needed for positivity of  the boundary integral (involving the Levi form) and the interior integral (involving the complex Hessian of the weight) (see \cite{Ho91}).
 Following \cite{varmcneal},	let us introduce  a notion from linear algebra.
	 	 \begin{defn}\label{def-qpositive}
	 	 	Let $H$ be a Hermitian form on a complex inner product space $(E,g)$. We say that $H$ is \emph{strictly $q$-positive} with respect to $g$ if the  sum of each collection of $q$ eigenvalues of $H$ with respect to $g$
	 	 	 is  positive. If the sum of each collection of $q$-eigenvalues is nonnegative, we say that $H$ is \emph{$q$-positive} with respect to $g$.
	 	 \end{defn}
 	  See below in Section~\ref{sec-eigenvalue} for a more detailed discussion of the eigenstructure of Hermitian forms with respect to an inner product.
 	 Generalizing the result of H\"ormander stated above, it follows from   \cite[Theorem~2.14]{varmcneal} that given a smoothly bounded relatively  compact domain
 	 $\Omega$ in a K\"ahler manifold, and a smooth weight $\varphi$ in a neighborhood of $\ol{\Omega}$,  simultaneous strict $q$-positivity of the Levi form of $\partial\Omega$
 	  and of the complex Hessian of $\varphi$ on $\ol{\Omega}$ constitutes a sufficient condition for the vanishing of the $L^2$-cohomology in degree $(p,q)$. We will prove  stronger versions of this result without the K\"ahler hypothesis below in Theorem~\ref{thm:L^2 theory, with weights}  and
 	  	Theorem~\ref{thm:L^2 theory, no weight}. Related results were studied in
 	  	\cite{Ho91,wu81} etc.

 	  The hypotheses and conclusions of Theorem~\ref{thm:L^2 theory, with weights}  and
 	  Theorem~\ref{thm:L^2 theory, no weight} have a dissatisfying incongruity about them and this is one of our motivations for starting this project.
	  The $L^2$-cohomology of a bounded domain in a complex manifold is defined independently of the choice of the Hermitian metric. On the other hand, the hypotheses on the Levi form of the boundary and
 	  the complex Hessian of the weight involve the metric (since strict $q$-positivity of the two Hermitian forms is defined with respect to this metric). Therefore, Theorem~\ref{thm:L^2 theory, with weights}  and
 	  Theorem~\ref{thm:L^2 theory, no weight} draw a metric-independent
 	  conclusion from a hypothesis that depends very much on the choice of a metric.  	   This paper is an attempt to understand what purely complex-geometric conditions on a domain in a complex manifold suffice to
 	  ensure that the $L^2$-cohomology is zero.
 	
 	\subsection{Results}   Andreotti and Vesentini in   \cite[Section 5]{av65} gave a
 	  proof of the vanishing theorem of Andreotti  and Grauert stated above by constructing a metric in which the complex Hessian of the exhaustion $\vp$ is
 	 strictly $q$-positive and applying $L^2$-methods.  We will use the same approach, but our metric construction is fundamentally different in that their metric is complete while we focus extensively on the interaction of the metric
	 with $\p\Om$. The subtle and delicate aspect of our work is ensuring that the Levi form and the complex Hessian of the weight function are simultaneously strictly $q$-positivity in the metric we construct. This allows us to
 	 apply $L^2$-results like
 	   Theorem~\ref{thm:L^2 theory, with weights}  and
 	  Theorem~\ref{thm:L^2 theory, no weight}.  Let us say that on a relatively compact domain $\Omega$ in a complex manifold, the $\dbar$-operator satisfies the \emph{Folland-Kohn basic estimate}
 	  in degree $(p,q)$,   if	  there exists a constant
 	  $C>0$ so that for all $(p,q)$-forms $f \in \Dom(\dbar)\cap\Dom(\dbars)$ we have
 	  \begin{equation}
 	  	\label{eq-follandkohn}
 	  	 	 \norm{f}_{L^2(\bd\Om)}^2 \leq C \left(\norm{\dbar f}_{L^2(\Om)}^2 + \norm{\dbars f}_{L^2(\Om)}^2 + \norm{f}_{L^2(\Om)}^2\right).
 	  	  \end{equation}
 	It is well-known that such an estimate has numerous important consequences for the function theory of $\Omega$. Combined with  the ellipticity of $\dbar\oplus\dbars$ in the interior of $\Om$ (where there are no boundary conditions), it follows that the \emph{ $L^2$-cohomology  space $H^{p,q}_{L^2}(\Om)$ is finite dimensional. }
 	Additionally, $\frac{1}{2}$-subelliptic estimates hold for the $\dbar$-Neumann problem (see \cite{CaSh07,FoKo72,Str10}).
 	
 	  Our techniques provide a new proof of the following well-known result.

\begin{thm}\label{thm-zq}Let $M$ be a complex manifold and
	let $\Omega\subset M$ be a smoothly bounded relatively compact domain, and let  $1 \leq q \leq n-1.$
	 Suppose that the domain $\Omega$ satisfies condition $Z(q)$.
	Then  the $\dbar$-operator on $(p,q)$-forms  on $\Omega$  satisfies the Folland-Kohn basic estimate \eqref{eq-follandkohn}.
 \end{thm}
Our main tool is the following proposition, which is closely related to
\cite[Lemma 18]{av65}, though our proof is quite different from that offered
in \cite{av65}. It has the added flexibility of prescribing the metric on a closed subset of the manifold.
\begin{prop}\label{prop-main}
	Let $E$ be a smooth complex vector bundle of rank $d$ over a smooth manifold $M$, and let $S$ be a smooth Hermitian form on $E$ such that  for some $1\leq \qt \leq d$, at each $p\in M$,  $S_p$ has at least $d-\qt+1$ strictly positive eigenvalues.
	Suppose that there is a (possibly empty) closed subset $F\subset M$ and a Hermitian metric $g_0$ on $M$ such that on $F$, the Hermitian form $S$ is strictly $\qt$-positive
	with respect to $g_0$. 	Then there is a Hermitian metric $g$ on $E$ such that
	$S$ is strictly $\qt$-positive with respect to $g$, and in a neighborhood of $F$ we have $g=g_0$.	
\end{prop}

As already noted above, the Folland-Kohn basic estimate \eqref{eq-follandkohn} proven  in Theorem~\ref{thm-zq} suffices to prove that the $L^2$ cohomology $H^{p,q} _{L^2}(\Om)$ of $\Omega$ is finite dimensional.  The main goal of this paper is to provide sufficient conditions for the \emph{vanishing} of this cohomology, in terms of
 global information about the embedding of $\Omega$ in the ambient manifold $M$.  This global structure is provided by a weight function which is compatible with the Levi-form of $\Omega$ in a sense which is made precise by the following theorem.  We use $T^{1,0}(\bd\Omega)$ to denote the $(1,0)$-tangent bundle   of the boundary of $\Omega$.

\begin{thm}\label{thm:continuously varying} Let $M$ be an $n$-dimensional complex manifold,  let $\Om\subseteq M$ be a smoothly bounded relatively compact domain, and  let
	$1 \leq q \leq n-1$.
 Suppose that there is a
	smooth function $\vp$ defined in a neighborhood of $\ol{\Omega}$ such that:\\[2mm]
	Either
	\begin{enumerate}
		\item 	 there is a continuous subbundle of rank $(n-q)$  of  $T^{1,0}(\bd\Omega)$ on which both the Levi form of $\bd\Om$  and the complex Hessian of $\vp$ are positive, and
		\item the complex Hessian of $\vp$ has at least $(n-q+1)$ positive eigenvalues at each point of $\ol{\Om}$.
	\end{enumerate}
Or
	\begin{enumerate}
		\item 	 there is a continuous subbundle of rank $(q+1)$  of  $T^{1,0}(\bd\Omega)$ on which both the Levi form of $\bd\Om$  and the complex Hessian of $\vp$ are negative,
		\item the complex Hessian of $\vp$ has at least $(q+1)$ negative eigenvalues at each point of $\ol{\Om}$, and
		\item the restriction of the complex Hessian of $\vp$ to $T^{1,0}(\partial\Omega)$ is nondegenerate at each point of $\partial\Omega$.
	\end{enumerate}
Then the $L^2$-cohomology $H^{p,q} _{L^2}(\Om)$ of $\Omega$ in degree $(p,q)$ vanishes for $0 \leq p \leq n$.
\end{thm}

Notice that the hypotheses of Theorem~\ref{thm:continuously varying} imply not only that the domain $\Omega$ satisfies condition $Z(q)$, but also that
of the two mutually exclusive conditions constituting $Z(q)$ (that the Levi form has $n-q$ positive or $q+1$ negative eigenvalues)  there is exactly one
which is satisfied at each point of the boundary.  It would be interesting to understand how to extend  to general $Z(q)$-domains
the techniques traditionally applied to annuli in $\cx^n$ or Stein manifolds (see \cite{shaw2010annuli, shaw2011annuli, LiShaw, ChHa20,ChHa21}).

In Theorem~\ref{thm:continuously varying},
the continuity of the subbundle of common positive directions of the two Hermitian forms may be a strong hypothesis.  Indeed, in Corollary \ref{cor:discontinuous_subbundle}, we will see that there exists an example of a smoothly parameterized family of Hermitian forms such that each form admits a positive eigenvalue but there does not exist a continuously parameterized vector field on which each Hermitian form is positive.  Fortunately, we can prove our result when $q=n-1$ without requiring continuity of the subbundle.  Indeed, we have the following.

\begin{thm}\label{thm:q_equals_n_minus_one} Let $M$ be an $n$-dimensional complex manifold and let $\Om\subseteq M$ be a smoothly bounded relatively compact domain.
 Suppose that there exists a
	smooth function $\vp$ defined in a neighborhood of $\ol{\Omega}$ such that
	\begin{enumerate}
		\item for every point $p\in\bd\Omega$, there exists a vector $L\in T_p^{1,0}(\bd\Omega)$ on which both the Levi form of $\bd\Om$  and the complex Hessian of $\vp$ are positive, and
		\item the complex Hessian of $\vp$ has at least $2$ positive eigenvalues at each point of $\ol{\Om}$.
	\end{enumerate}
Then the $L^2$ cohomology $H^{p,n-1} _{L^2}(\Om)$ of $\Omega$ in degree $(p,n-1)$ vanishes for $0 \leq p \leq n$.
\end{thm}

Notice that in Theorems~\ref{thm-zq}, \ref{thm:continuously varying} and \ref{thm:q_equals_n_minus_one},     the hypotheses
are independent of a choice of metric, i.e., they are determined solely
by the complex structure. In particular, the $L^2$ spaces of forms on $\Om$
are defined with respect to \emph{any} Hermitian metric on $M$. Notice that though the inner product and norm of $L^2_{p,q}(\Omega)$ depend on the metric chosen on $M$, the space $L^2_{p,q}(\Omega)$ itself is determined independently of the choice of the Hermitian metric on $M$. It follows that
the ``maximally realized" $\dbar$-operator $\dbar: L^2_{p,q-1}(\Om)\to L^2_{p,q}(\Om)$ is also defined
independently of the choice of the metric, as is the corresponding $L^2$-cohomology $H^{p,q}_{L^2}(\Omega)$
of the domain $\Omega$.


\section{Definitions and preliminaries}
\subsection{Hermitian forms and Eigenvalues}
\label{sec-eigenvalue}
   Recall that a \emph{Hermitian form} on a
complex vector space $E$ is a map $H:E\times E\to \cx$ such that for $u,v,w\in E$ and $a\in \cx$ we have  $H(au+v,w)=aH(u,\violet{w})+ H(v,w)$, and
$H(u,v)=\ol{H(v,u)}$. The
 Hermitian form  $H$ is a \emph{Hermitian metric} or \emph{inner product} if it is \emph{positive definite}\violet{:} $H(v,v)>0$. According to \emph{Sylvester's law of inertia},  we can write
\[ H(z,z)=\sum_{j=1}^{\red d} \epsilon_j \abs{\ell_j(z)}^2, \]
where $\epsilon_j\in \{1, -1, 0\}$, $\ell_j:E\to \cx$ are linearly independent linear forms, and $d$ is the dimension of $E$.  The 3-tuple of integers  counting respectively the number of
positive, negative and zero coefficients among the $\epsilon_j$'s is known as the
\emph{inertia} or \emph{signature} of the form and is an invariant completely classifying Hermitian forms on $E$ up to linear automorphisms.  By a standard abuse of language we will
refer to the integers constituting the inertia as the number of positive, negative and zero \emph{eigenvalues} of $H$.

     Given a Hermitian metric $g$ on $E$, and a Hermitian form $H$ on $E$, recall that $v\in E$ is an \emph{eigenvector of $H$ with eigenvalue $\lambda\in\mathbb{R}$ with respect to $g$} if $H(v,u)=\lambda g(v,u)$ for all $u\in E$. Equivalently, we may define an operator $H^g:E\to E$ by
 \[ g(H^gu,v)= H(u,v) \quad \text{ for all } u,v\in E,\]
and the eigenpairs of the Hermitian form $H$ with respect to $g$ will correspond to the eigenpairs of the operator $H^g$. It is clear that $H^g:E\to E$ is a Hermitian operator (with respect to the metric $g$), and consequently, 
its eigenvalues are real, and the eigenvectors can be taken to be orthogonal with respect to the metric.
We will denote by $\lambda_j^g(H)$ the $j$-th smallest eigenvalue of $H^g$, where eigenvalues are counted
with multiplicity. Therefore
\begin{equation}\label{eq-eigens}
	\lambda_1^g(H)\leq \lambda_2^g(H)\leq \dots \leq \lambda_d^g(H).
\end{equation}
From the spectral theorem for each nonzero vector $z\in E$ we have the following:
\begin{equation}\label{eq-rayleigh}
	\lambda_{1}^g(H) \leq  \frac{H(z,z)}{g(z,z)} \leq \lambda_d^g(H),
\end{equation}
with equality if $z$ is an eigenvector of the corresponding eigenvalues. This follows by writing $  {H(z,z)}=g(H^gz,z)$ in terms of a basis of orthonormal eigenvectors of $H^g$. 
Notice that a Hermitian form $H$ is strictly $q$-positive with respect to the metric $g$ if and only if
\begin{equation}\label{eq-qpositive}
	\sum_{k=1}^{q} \lambda^g_k(H)>0,
\end{equation}
i.e., the sum of the \emph{smallest} $q$ eigenvalues is positive.


A Hermitian form (resp., metric) on a complex vector bundle on a smooth manifold is the assignment of a Hermitian form  (resp., metric)  to each fiber.
Given a Hermitian form $H$ and a Hermitian metric $g$
on a vector bundle, we say that $H$ is (strictly) $q$-positive  with respect to $g$ if
it is  so on each fiber with respect to the metric on that fiber.

Recall that given a Hermitian form $H$ and a metric $g$ on a vector space $V$, the trace of $H$ with respect to $g$ is the sum of all the eigenvalues of $H$ with respect to $g$:
\[ \tr_g(H)= 	\sum_{k=1}^{\dim V} \lambda^g_k(H).\]
It is well-known that
\[ \tr_g(H)= 	\sum_{k=1}^{\dim V} H(t_k, t_k)\]
for each orthonormal basis $\{t_k\}$ of $V$. The following
characterization of strict $q$-positivity is classical and is a consequence
of the Schur Majorization theorem (\cite[Theorem~4.3.45]{HoJo13} and  \cite{ikebe}; see also  \cite[Lemma 4.7]{Str10}).
\begin{thm}\label{thm-schur}
	A Hermitian form on a finite dimensional inner-product space  is strictly $q$-positive if and only if its restriction to each $q$-dimensional linear subspace has positive trace.
\end{thm}
\subsection{The Levi form }\label{sec:levi}
It is possible to define the Levi form  of a hypersurface
fully intrinsically, with values in the line bundle of bad directions (see \cite{leviform}). For our purposes,
it will suffice to use a definition of the Levi form in terms of a defining function.

Let $\Omega$ be a smoothly bounded and relatively compact domain in a
complex manifold $M$, and let $\rho$ be a defining function of $\Omega$, i.e.,
$\rho$ is a smooth function in a neighborhood $U$ of  $\partial \Omega$ such that $\{\rho<0\}=U\cap \Omega$ and $d\rho\not=0$ on $\partial \Omega$. One then
defines the Levi form of $\partial\Omega$ as the Hermitian form on $T^{1,0}(\partial \Om)$ given by
\begin{equation}
	\label{eq-lrho}\opL_\rho(X,Y)= \partial \dbar \rho(X,\ol{Y}), \quad X,Y\in T_p^{1,0}(\partial\Omega).
\end{equation}
The defining function of the domain $\Omega$ is not unique, but if $r$ is another defining function, then there is a smooth function $f>0$ defined
near $\partial \Omega$ such that $\rho=f\cdot r$. It then easily follows that
\[\opL_\rho=f \opL_r,\]
so that at each $p\in \partial \Omega$, for $X\in T_p^{1,0}(\partial \Omega)$
the real number $\opL_\rho(X,X)$ is positive (resp., negative, resp., zero)
if and only if $\opL_r(X,X)$ is positive (resp., negative, resp., zero).
We see therefore that the conditions on the Levi form in the hypothesis of
Theorem~\ref{thm:continuously varying} are invariantly defined independently
of the choice of the defining function.

In Theorem~\ref{thm:L^2 theory, with weights} below, we are given in addition to the domain $\Omega$, a
Hermitian metric $g$ on the manifold $M$.  Since $\mathcal{L}_\rho=f\mathcal{L}_r$, $\mathcal{L}_\rho$ and $\mathcal{L}_r$ have the same collection of eigenvectors in $T^{1,0}_p(\partial\Omega)$ with respect to $g$,
and if $\{\lambda_j(p)\}$ are the eigenvalues of $\opL_\rho(p)$, then the eigenvalues of $\opL_r(p)$ are clearly $\{f(p)\cdot \lambda_j(p)\}$.
Since $f>0$, it follows that the condition of strict $q$-positivity in the hypotheses of Theorems~\ref{thm:L^2 theory, with weights}  and \ref{thm:L^2 theory, no weight} are invariantly defined independently of the choice of defining function.

\section{Some results from \texorpdfstring{$L^2$}{L2}-theory}

\subsection{Unweighted \texorpdfstring{$L^2$}{L2} result} The following will be needed in the proof of Theorem~\ref{thm-zq} and is proven in Section \ref{subsec:proof no weight L^2}.
\begin{thm}\label{thm:L^2 theory, no weight} Let $M$ be a complex manifold, $\Om\subseteq M$ a relatively compact $C^3$ domain with defining function $\rho$,  $0\leq p \leq n$, and $1 \leq q \leq n-1$.
	Let $h$ be a Hermitian metric on $M$ such that on each connected component of $\bd\Om$, either
	the Levi form, $\opL_\rho^h$ is  strictly $q$-positive or the negative of the Levi form $-\opL_\rho^h$ is strictly $(n-q-1)$-positive.
	Then  the $\dbar$ operator on $(p,q)$-forms on $\Omega$ satisfies the Folland-Kohn basic estimate   \eqref{eq-follandkohn}.
	%
\end{thm}

\begin{rem}
It is well known that $C^2_{p,q}(\bar\Om) \cap \Dom(\dbars)$ is dense in $\Dom(\dbar)\cap\Dom(\dbars)$ (see, e.g.,  \cite[p.121]{Hor65})
so it suffices to assume that our forms have coefficients that \violet{are} at least $C^2$ smooth on $\bar\Om$.
Further, by the comments at the end of Section~\ref{sec:levi}, the fact that $\opL_\rho^h$ is strictly $q$-positive is independent of the choice of the defining function $\rho$, so the hypotheses do not depend on the choice of $\rho$.
\end{rem}

The function theory in \cite[Section 5.3]{ChSh01} is largely applicable to our setup, but we have to adapt their work to domains
that are not necessarily pseudoconvex. Our analysis combines elements of
\cite{ChSh01} with ideas from \cite{Ho95} and \cite{HaRa15}.  The latter papers, however, were concerned with domains in $\C^n$ and Stein manifolds, respectively.
See also \cite[\S1.9]{Zam08} for a related discussion but for domains in $\C^n$,
and \cite[Theorem~2.14]{varmcneal} for a similar result in K\"ahler manifolds.

\subsection{Weighted \texorpdfstring{$L^2$}{L2} result}
The following is the key result that will be used to prove Theorem~\ref{thm:continuously varying} and its proof appears in Section \ref{subsec:proof of weighted L^2}.
\begin{thm}\label{thm:L^2 theory, with weights} Let $(M,h)$ be a Hermitian manifold of complex dimension $n$, let $\Om\subseteq M$ be a relatively compact domain with $C^3$ boundary $\partial \Omega$ and defining function $\rho$,  let $\varphi$ be a smooth function in a neighborhood of $\ol{ \Omega}$ and let $1 \leq q \leq n-1$. Suppose that, 	with respect to $h$:\\[2mm]
	 Either
	\begin{enumerate}
		\item $\mathcal{L}_\rho^h$  is  $q$-positive on $\partial \Omega$, and
\item on $\overline\Omega$,	the complex Hessian of $\varphi$ is  strictly $q$-positive.
	\end{enumerate}
Or
	\begin{enumerate}
		\item $-\mathcal{L}_\rho^h$  is  $(n-q-1)$-positive on $\partial \Omega$,
\item on $\overline\Omega$,	the complex Hessian of $-\varphi$ is  strictly $(n-q)$-positive, and
\item if $L_n\in T^{1,0}(M)$ is orthogonal to $T^{1,0}(\partial\Omega)$ and unit length on $\partial\Omega$ with respect to $h$, then the sum of any $(n-q)$ eigenvalues of the complex Hessian of $-\varphi$ is strictly larger than $-\ddbar\varphi(L_n,\bar L_n)$ on $\partial\Omega$.
	\end{enumerate}
	Then for each $0\leq p \leq n$ and $t$ sufficiently large, there exists a constant
	$C>0$ so that for all $(p,q)$-forms $f \in \Dom(\dbar)\cap\Dom(\dbars)$
\begin{equation}
\label{eq:basic_estimate}
\|f\|_{L^2(\Om,e^{-t\vp})}^2 \leq C \left(\|\dbar f\|_{L^2(\Om,e^{-t\vp})}^2 + \|\dbars_{t\vp} f\|_{L^2(\Om,e^{-t\vp})}^2\right).
\end{equation}
\end{thm}
\subsection{Preliminaries for  \texorpdfstring{$L^2$}{L2} methods} Suppose that $\vp$ is a smooth  function
that is defined on a neighborhood of $\ol{\Om}$.
Let $L^2_{p,q}(\Om,t\vp)$ denote the $L^2$-space of $(p,q)$-forms on $\Om$ with norm
\[
\|f\|_{L^2(\Om,t\vp)}^2 = \int_\Om |f|^2\, e^{-t\vp}\, dV
\]
and corresponding inner product.
Denote the $L^2$-adjoint of $\dbar:L^2_{p,q}(\Om,t\vp) \to L^2_{p,q+1}(\Om,t\vp)$ by
$\dbars_{t\vp}:L^2_{p,q+1}(\Om,t\vp) \to L^2_{p,q}(\Om,t\vp)$.
 It is well known that $\Dom(\dbars_{t\vp})$ does not depend on $t\vp$.

We assume that $\rho$ is a defining function for $\Om$ so that $|d\rho|=1$ on $\bd\Om$. Let $U$ be a open set that intersects $\bar\Om$
and admits a basis of orthonormal $(0,1)$-forms $\om_1,\dots,\om_n$ so that $\om_n = \p\rho$, and
\begin{equation}\label{eqn:C-S setup for d-dbar-rho}
\p\dbar\rho = \overline{\dbar\om_n} = \sum_{j,k=1}^n \rho_{jk}\, \om_j\wedge\omb_k
\end{equation}
where $(\rho_{jk})$ is the Levi matrix.
We denote by $\bar L_1,\dots,\bar L_n$ the basis for $T^{1,0}(U)$ that is dual to $\omb_1,\dots,\omb_n$.  For each $1\leq j\leq n$, set $\delta_j^{t\varphi}=e^{t\varphi}L_j e^{-t\varphi}$, so that the adjoint of $\bar L_j$ with respect to $L^2(\Omega,e^{-t\varphi})$ is given by $-\delta_j^t$ up to a zero-order term that is independent of $t$ and $\varphi$.

Let $\opI_q = \{J=(j_1,\dots,j_q) \in \N^q : 1 \leq j_1 < \cdots < j_q \leq n\}$ denote the set of increasing $q$-tuples. We can then express a $(0,q)$ form $f$ on $U$ by
\[
f = \sum_{J \in \I_q} f_J\, \omb^J.
\]
Given $I \in \I_{q-1}$ and $J\in\I_q$, we define
$\ep^{jI}_K$ to be $0$ if $\{j\}\cup I \neq J$ as sets  and otherwise is the sign of the permutation that reorders $jI$ as $K$. We also employ the standard
notation
\[
f_{jI} = \sum_{J \in \I_q} \ep^{jI}_J f_J
\]
and use the shorthand $\|\bar L f\|_{L^2(\Om,e^{-\vp})}$ for
\[
\|\bar L f\|_{L^2(\Om,e^{-\vp})}^2 =  \sum_{J \in \I_q}\sum_{k=1}^n \|\bar L_k f_J\|_{L^2(\Om,e^{-\vp})}^2.
\]

The key to proving Theorem \ref{thm:L^2 theory, no weight} (and  Theorem \ref{thm:L^2 theory, with weights}) is to establish a good \emph{basic identity}, and this is provided by
equation (5.3.20) in \cite[Section 5.3]{ChSh01}.  For forms $f \in \Dom(\dbars_{t\vp})\cap C^2_{0,q}(\bar\Om)$ supported in $U$ where $U$ is open set that admits good local coordinates,
\begin{align}
\|\dbar f\|_{L^2(\Om,e^{-t\vp})}^2 + \|\dbars_{t\vp} f\|_{L^2(\Om,e^{-t\vp})}^2
&= \|\bar L f\|_{L^2(\Om,e^{-t\vp})}^2 + t\sum_{I \in \I_{q-1}} \sum_{j,k=1}^n (\vp_{jk} f_{jI},f_{kI})_\vp \nn \\
&+ \sum_{I \in \I_{q-1}} \sum_{j,k=1}^n \int_{\bd\Om \cap U} \rho_{jk} f_{jI} \overline{f_{kI}} e^{-t\vp}\, d\sigma
+ R(f) + E(f) \label{eqn:basic identity}
\end{align}
where $E(f)$ satisfies both
\begin{equation}
\label{eq:E_standard_estimate}
|E(f)| \leq C\|\bar L f\|_{L^2(\Om,e^{-t\vp})}\|f\|_{L^2(\Om,e^{-t\vp})}
\end{equation}
and via integration by parts of the tangential terms,
\begin{equation}
\label{eq:E_IBP_estimate}
|E(f)| \leq C\Big(\sum_{J \in\I_q} \|\bar L_n f_J\|_{L^2(\Om,e^{-t\vp})} + \sum_{J\in\I_q}\sum_{j=1}^{n-1} \|\delta_j^{t\vp} f_J\|_{L^2(\Om,e^{-t\vp})}\Big)\|f\|_{L^2(\Om,e^{-t\vp})}
\end{equation}
and for any $\ep>0$, there exists $C_\ep>0$ so that
\begin{equation}\label{eqn:R(f)}
|R(f)| \leq \ep(\|\dbar f\|_{L^2(\Om,e^{-t\vp})}^2 + \|\dbarsvp f\|_{L^2(\Om,e^{-t\vp})}^2) + C_\ep \|f\|_{L^2(\Om,e^{-t\vp})}^2.
\end{equation}
Additionally, the constants $C$ and $C_\ep$ are independent of $t$ and $\vp$. The basic identity \eqref{eqn:basic identity} is precisely the result we need
to prove Theorem \ref{thm:L^2 theory, with weights} and Theorem \ref{thm:L^2 theory, no weight} for the case when $\opL_\rho^h$ is $q$-positive
on $U\cap\bd\Om$.

We now establish a basic identity for the components of $\bd\Om$ for which $-\opL_\rho^h$ is $(n-1-q)$-positive.
Following \cite{ChSh01}, we have
\[
[\delta_k^{t\vp},\bar L_k]u
= \phi_{kk} u +  \sum_{\ell=1}^n c_{kk}^\ell \delta_\ell^{t\vp}(u) - \sum_{\ell=1}^n \bar c_{kk}^\ell \bar L_\ell(u)
\]
where $\rho_{kk} = c_{kk}^n = \bar c_{kk}^n$, and $c^\ell_{kk}$ are $C^1$ functions on $U$. Integration by parts then yields
\begin{align}\label{eqn:bar L_k f_J IBP}
\|\bar L_k f_J\|_{L^2(\Om,e^{-t\vp})}^2
&= \|\delta_k^{t\vp} f_J\|_{L^2(\Om,e^{-t\vp})}^2 - t(\vp_{kk} f_J,f_J)_{t\vp}
- \Big(\sum_{\ell=1}^n c_{kk}^\ell \delta_\ell^{t\vp} f_J, f_J\Big)_{t\vp} \\ &+ O(\|\bar L f\|_{L^2(\Om,e^{-t\vp})}\|f\|_{L^2(\Om,e^{-t\vp})}). \nn
\end{align}
Integration by part shows that if $\ell < n$, then
\[
|(c_{kk}^\ell \delta_\ell^{t\vp} f_J, f_J)_{t\vp}| \leq C \|\bar L f\|_{L^2(\Om,e^{-t\vp})}\|f\|_{L^2(\Om,e^{-t\vp})}.
\]
However, if $\ell=n$, then
\[
(c^n_{kk} \delta_n^{t\vp}f_J,f_J)_{t\vp} = \int_{\bd\Om\cap U} \rho_{jk} |f_J|^2 e^{-t\vp}\, d\sigma
+ O(\|\bar L f\|_{L^2(\Om,e^{-t\vp})}\|f\|_{L^2(\Om,e^{-t\vp})}).
\]
Thus, we have now established \cite[(3.20)]{Sha85a} but for forms $f \in \Dom(\dbar)\cap\Dom(\dbars)$ supported in a suitably small neighborhood in a complex
manifold (vs. in $\C^n$), namely,
\begin{align}
\|\dbar f\|_{L^2(\Om,e^{-t\vp})}^2 + \|\dbars_{t\vp} f\|_{L^2(\Om,e^{-t\vp})}^2
&= \sum_{J \in\I_q} \|\bar L_n f_J\|_{L^2(\Om,e^{-t\vp})}^2 + \sum_{J\in\I_q}\sum_{j=1}^{n-1} \|\delta_j^{t\vp} f_J\|_{L^2(\Om,e^{-t\vp})}^2 \nn \\
&+ t \sum_{I \in\I_{q-1}}\sum_{j,k=1}^n ( \vp_{jk} f_{jI},f_{kI})_{t\vp} - t\sum_{J\in\I_q}\sum_{j=1}^{n-1} (\vp_{jj} f_J,f_J)_{t\vp}\nn  \\
&+ \sum_{I\in\I_{q-1}} \sum_{j,k=1}^n \int_{\bd\Om} \rho_{jk} f_{jI} \overline{f_{kI}} e^{-t\vp}\, d\sigma
- \sum_{J \in\I_q}\sum_{j=1}^{n-1}\int_{\bd\Om} \rho_{jj} |f_J|^2 e^{-t\vp} \, d\sigma   \label{eqn:basic identity, IBP}\\
&+E(f)+R(f).  \nn
\end{align}
Notice that our coordinates are such that
$\displaystyle{
\sum_{j=1}^{n-1} \rho_{jj} = \Tr(\opL^h_\rho).}
$

\subsection{The proof of Theorem \ref{thm:L^2 theory, no weight}}\label{subsec:proof no weight L^2}
We use the basic identities (\ref{eqn:basic identity}) and \eqref{eqn:basic identity, IBP} with $\vp=0$. In the case that $\opL_\rho^h$ is strictly
$q$-positive on $U\cap \bd\Om$,  it follows from standard multilinear algebra (see, e.g., \cite[Lemma 4.7]{Str10}) that there exists $c>0$
so that for $f \in C^2_{p,q}(\bar\Om)\cap\Dom(\dbars)$ supported in $U$,
(\ref{eqn:basic identity}) becomes
\begin{align*}
\|\dbar f\|_{L^2(\Om)}^2 + \|\dbars f\|_{L^2(\Om)}^2
&\geq \|\bar L f\|_{L^2(\Om)}^2 + c \|f\|_{L^2(\bd\Om)}^2
+ R(f) + E(f).
\end{align*}
A small constant/large constant argument and an absorption of terms establishes the Folland-Kohn basic estimate for $f$ supported on $U$
when $\opL_\rho^h$ is strictly $q$-positive on $U\cap\bd\Om$.

Now suppose that $f$ is supported in a neighborhood $U$ so that $-\opL_\rho^h$ is $(n-1-q)$-positive on $\bd\Om\cap U$.
In this case, (\ref{eqn:basic identity, IBP}) with $\vp=0$ becomes
\begin{align}
\|\dbar f\|_{L^2(\Om)}^2 + \|\dbars f\|_{L^2(\Om)}^2
&= \sum_{J \in\I_q} \|\bar L_n f_J\|_{L^2(\Om)}^2 + \sum_{J\in\I_q}\sum_{j=1}^{n-1} \|L_j f_J\|_{L^2(\Om)}^2 \nn \\
&+ \sum_{I\in\I_{q-1}} \sum_{j,k=1}^n \int_{\bd\Om} \rho_{jk} f_{jI} \overline{f_{kI}}\, d\sigma
- \sum_{J \in\I_q}\int_{\bd\Om} \Tr(\opL) |f_J|^2  \, d\sigma   \label{eqn:basic identity, IBP, no weight}\\
&+E(f)+R(f)  \nn
\end{align}
For the boundary integral, we note that
\begin{equation}\label{eqn:-opL positivity}
\lambda_1^h + \cdots + \lambda_q^h - (\lambda_1^h + \cdots \lambda_{n-1}^h)
= - (\lambda_{q+1}^h+\cdots + \lambda_{n-1}^h)>0
\end{equation}
since $-\opL_\rho^h$ is $(n-1-q)$-positive. This means that
a small constant-large constant argument and an absorption of the $E(f)$ and $R(f)$ terms shows that the Folland-Kohn basic estimate holds
in this regime as well.

A partition of unity argument shows that the Folland-Kohn basic estimate holds globally and Theorem \ref{thm:L^2 theory, no weight} is proved.

\subsection{The proof of Theorem \ref{thm:L^2 theory, with weights}}\label{subsec:proof of weighted L^2}
We first consider the case in which the Levi-form of the boundary is $q$-positive and the complex Hessian of $\varphi$ is strictly $q$-positive. We cover $\bar\Om$ with good neighborhoods $\{U_k\}$ and use a partition
of unity subordinate to $\{U_k\}$ so that \eqref{eqn:basic identity} applies on each $U_k$.
As in the proof of Theorem \ref{thm:L^2 theory, no weight}, we may use, e.g., \cite[Lemma 4.7]{Str10} to show that the boundary integral term in \eqref{eqn:basic identity} is non-negative.  We then use a small constant/large constant argument with \eqref{eq:E_standard_estimate} to absorb  $\|\bar L f\|_{L^2(\Om,e^{-t\vp})}$ and take
$\epsilon$ small enough to absorb the $\dbar$ and $\dbars_{t\vp}$ terms in \eqref{eqn:R(f)}. Finally, we use the strict $q$-positivity of the complex Hessian of $\vp$ and take $t$ large enough to absorb the remaining error terms. Consequently,
there exists $C>0$ so that \eqref{eq:basic_estimate} follows.

Suppose, on the other hand, that $-\mathcal{L}_\rho^h$ is $(n-q-1)$-positive and the complex Hessian of $-\varphi$ is strictly $(n-q)$-positive.  Choose a neighborhood $U_k$ on which \eqref{eqn:basic identity, IBP} holds.  Using a small constant/large constant argument with \eqref{eq:E_IBP_estimate} to absorb terms with derivatives in \eqref{eqn:basic identity, IBP} and similar estimates to absorb $R(f)$ gives us
\begin{multline}
\label{eqn:basic identity, IBP, weight}
\|\dbar f\|_{L^2(\Om,e^{-t\vp})}^2 + \|\dbars_{t\vp} f\|_{L^2(\Om,e^{-t\vp})}^2
\geq C t \sum_{I \in\I_{q-1}}\sum_{j,k=1}^n ( \vp_{jk} f_{jI},f_{kI})_{t\vp} - C t\sum_{J\in\I_q}\sum_{j=1}^{n-1} (\vp_{jj} f_J,f_J)_{t\vp}\\
+ C\sum_{I\in\I_{q-1}} \sum_{j,k=1}^n \int_{\bd\Om} \rho_{jk} f_{jI} \overline{f_{kI}} e^{-t\vp}\, d\sigma
-C \sum_{J \in\I_q}\sum_{j=1}^{n-1}\int_{\bd\Om} \rho_{jj} |f_J|^2 e^{-t\vp} \, d\sigma-C\|f\|_{L^2(\Om,e^{-t\vp})}^2,
\end{multline}
where $C>0$ is independent of $f$ and $t$.  As before, \eqref{eqn:-opL positivity} implies that the boundary integral is non-negative.  Denote the eigenvalues of the complex Hessian of $\varphi$ with respect to $h$ in non-decreasing order by $\{\mu_1^h,\ldots,\mu_n^h\}$.  By the usual multilinear algebra,
\[
  \sum_{I \in\I_{q-1}}\sum_{j,k=1}^n ( \vp_{jk} f_{jI},f_{kI})_{t\vp} - \sum_{J\in\I_q}\sum_{j=1}^{n-1} (\vp_{jj} f_J,f_J)_{t\vp}\geq \epsilon\|f\|_{L^2(\Om,e^{-t\vp})}^2
\]
for some $\epsilon>0$ if
\[
  \sum_{j=1}^q\mu_j^h-\sum_{j=1}^n\mu_j^h+\vp_{nn}=-\sum_{j=q+1}^n\mu_j^h+\vp_{nn}
\]
is strictly positive on some interior neighborhood of the boundary.  Our hypotheses guarantee that this holds on $\partial\Omega$, so it must hold on some neighborhood of $\partial\Omega$, and hence we may choose $t$ sufficiently large in \eqref{eqn:basic identity, IBP, weight} to obtain \eqref{eq:basic_estimate} for forms supported on some interior neighborhood of the boundary.

On any neighborhood $U_k$ which does not intersect $\partial\Omega$, we may integrate by parts in $\{L_j\}_{1\leq j\leq n}$ for any orthonormal basis for $T^{1,0}(\partial\Omega)$.  Following the same procedure used to obtain \eqref{eqn:basic identity, IBP, weight}, we obtain
\begin{multline*}
\|\dbar f\|_{L^2(\Om,e^{-t\vp})}^2 + \|\dbars_{t\vp} f\|_{L^2(\Om,e^{-t\vp})}^2
\geq C t \sum_{I \in\I_{q-1}}\sum_{j,k=1}^n ( \vp_{jk} f_{jI},f_{kI})_{t\vp} - C t\sum_{J\in\I_q}\sum_{j=1}^{n} (\vp_{jj} f_J,f_J)_{t\vp}\\
-C\|f\|_{L^2(\Om,e^{-t\vp})}^2,
\end{multline*}
for any form supported in $U_k$.  Here, it suffices to observe that the complex Hessian of $-\varphi$ is $(n-q)$-positive to obtain \eqref{eq:basic_estimate}.  Increasing the size of $t$ as needed to accommodate error terms arising from our partition of unity, we obtain \eqref{eq:basic_estimate} globally.

\section{Constructing a metric for a single Hermitian form}

\subsection{Analyticity of spectral projections}
Let $n$ be a positive integer and  let $\mathscr{H}_n$ denote the real vector space of $n\times n$ Hermitian matrices.  For a matrix $T\in \mathscr{H}_n$, let $\sigma(T)\subset \mathbb{R}$ be its spectrum, i.e., the set of eigenvalues, and   for a subset $G\subset \mathbb{C}$ denote by $\pi_G(T)$ the matrix of the orthogonal projection from $\mathbb{C}^n$
onto the direct sum of  the  eigenspaces of $T$ corresponding to eigenvalues in the set $G$.  We have
\[(\pi_G(T))^2=\pi_G(T), \quad (\pi_G(T))^*=\pi_G(T)\] and  \[ \mathrm{range}\,\left( \pi_G(T)\right)= \bigoplus_{\lambda\in G\cap \sigma(T)}\{x\in \mathbb{C}^n: Tx=\lambda x\}.\]
If $G\cap \sigma(T)=\emptyset$ we let $\pi_G(T)=0$, consistent with the convention that an empty internal direct sum of vector subspaces is the zero subspace. We have the following.
\begin{prop}\label{prop-analytic}
	Let $G\subset \mathbb{C}$ be a smoothly bounded open subset of the complex plane, and let ${\mathscr{W}_G}\subset \mathscr{H}_n$ be the set of $n\times n$ Hermitian matrices which have no eigenvalues on the boundary $\bd G$:
	\[{\mathscr{W}_G}=\left\{T\in \mathscr{H}_n: \sigma(T)\cap \bd G =\emptyset \right\}. \]
	Then the mapping
	\begin{equation}\label{eq-pitu}
		\pi_G:{\mathscr{W}_G}\to \mathrm{Mat}_{n\times n}(\mathbb{C})
	\end{equation}
	is real analytic.
	
\end{prop}

\begin{proof}
	This is a consequence of the following classical representation of the matrix $\pi_G(T)$ as a contour integral (see \cite{kato}):
	\begin{equation}
		\label{eq-resolvent}
		\pi_G(T)= \frac{1}{2\pi i} \int_{\bd G} (\zeta I-T)^{-1}d\zeta,
	\end{equation}
	where
	\begin{enumerate}[wide, label=(\alph*)]
		\item $\zeta\mapsto (\zeta I-T)^{-1}\in \mathrm{Mat}_{n\times n}(\mathbb{C})$ is a matrix valued holomorphic function defined in a neighborhood of $\bd G$. What we mean by this is that each entry of the output matrix is holomorphic in $\zeta$, something that is
		obvious from Cramer's rule and the fact that for each $\zeta\in \bd G$, the matrix $(\zeta I-T)$ is invertible 
		(by the hypothesis that $T\in {\mathscr{W}_G}$).
		
		\item The integral is a line integral, taken entry-wise. The contour $\bd G$ is given the standard orientation coming from the orientation of $G$ as an open subset of $\mathbb{C}$.
		
	\end{enumerate}
	It then follows that the mapping \eqref{eq-pitu} is real analytic. Indeed, if we allow $T$ to have complex entries, differentiation under the integral sign in formula \eqref{eq-resolvent} shows that the mapping \eqref{eq-pitu} extends to a holomorphic
	mapping of matrices.

	To complete the proof, we need to establish the formula \eqref{eq-resolvent}. Fix $T\in {\mathscr{W}_G}\subset \mathscr{H}_n$, and note that $(\zeta I-T)$ is an invertible
	matrix provided $\zeta\not \in \sigma(T)$. This means that if $G$ does not contain any eigenvalue of $T$, then the function $\zeta\mapsto (\zeta I-T)^{-1}$ is holomorphic in a neighborhood of $\overline{G}$, i.e., each entry of this matrix-valued function is holomorphic in $\zeta$.  Therefore by Cauchy's theorem, in this case we have $\pi_G(T)=0$, as needed.
	
	Now consider the situation where $\sigma(T)=\{\lambda_1,\dots, \lambda_n\}$ and for some $1\leq k \leq n$ we have $\lambda_1, \dots \lambda_k\in G$ and the rest of the eigenvalues $\lambda_{k+1},\dots, \lambda_n\in \mathbb{C}\setminus \overline{G}$. Let $\{v_1,\dots, v_n\}$ be an orthonormal set of eigenvectors corresponding to
	these eigenvalues, i.e.,  $Tv_j=\lambda_j v_j, \norm{v_j}=1$ for each $j$ and $\langle v_j, v_k\rangle =\delta_{jk}$ for $1\leq j, k \leq  n$.
	Let $U$ be the unitary automorphism of
	$\mathbb{C}^n$ defined by $Ue_j=v_j$, where $\{e_j, 1\leq j \leq n\}$ is the standard basis of $\mathbb{C}^n$. Then it is clear that to prove the result, it suffices to prove that if $P$ is the matrix defined by the integral on the right hand side
	of \eqref{eq-resolvent} then $U^{-1}P U$ coincides with the matrix of the orthogonal projection onto the subspace spanned by the vectors $\{e_1,\dots, e_k\}$, i.e.\violet{,}
	\[ U^{-1}P U = \begin{pmatrix}
		I_{k\times k} &0\\0 & 0_{n-k\times n-k}
	\end{pmatrix}.\]
	Observe that
	\[U^{-1}(\zeta I- T)^{-1}U= \left[U^{-1} (\zeta I-T)U \right]^{-1}=(\zeta I- U^{-1}T U)^{-1}= \mathrm{diag}\left( \frac{1}{\zeta-\lambda_1}, \dots, \frac{1}{\zeta-\lambda_n}\right),\]
	where we use the fact that $ U^{-1}TU= \mathrm{diag}(\lambda_1,\dots, \lambda_n).$
	By the linearity of the integral we have
	\begin{align*}
		U^{-1}P U &= \frac{1}{2\pi i} \int_{\bd G} U^{-1}(\zeta I-T)^{-1}Ud\zeta\\
		&=  \frac{1}{2\pi i} \int_{\bd G} \mathrm{diag}\left( \frac{1}{\zeta-\lambda_1}, \dots, \frac{1}{\zeta-\lambda_n}\right)d\zeta\\
		&= \mathrm{diag}\left( \frac{1}{2\pi i} \int_{\bd G} \frac{d\zeta}{\zeta-\lambda_1}, \dots,  \frac{1}{2\pi i} \int_{\bd G} \frac{d\zeta}{\zeta-\lambda_n}  \right)\\&=  \begin{pmatrix}
			I_{k\times k} &0\\0 & 0_{n-k\times n-k}
		\end{pmatrix},
	\end{align*}
	since $\lambda_1,\dots, \lambda_k\in G$ and $\lambda_{k+1}, \dots \lambda_n\in \mathbb{C}\setminus \overline{G}$.
\end{proof}
\subsection{Proof of Proposition~\ref{prop-main}}  Let $h$ be a Hermitian metric on $E$,  and for $1\leq j \leq d$, let $ \lambda_j^h(p)$
		denote the $j$-th smallest eigenvalue of the  Hermitian form $S_p$ with respect to the metric $h$ (counting with multiplicity)
		so that
		\[ \lambda_1^h(p)\leq \lambda_2^h(p) \leq \dots \leq \lambda_d^h(p). \]
Since the $j$-th smallest eigenvalue of a matrix depends continuously on its entries (see \cite{kato}) it follows that
	the real-valued function $p\mapsto  \lambda_j^h(p)$ is continuous
	on $M$ for each $j$. Denote by $\nu_+(p)$ (resp., $\nu_-(p)$)
	 the number of positive (resp., negative)  eigenvalues of
	$S_p$, which is well-known not to depend on the metric $h$.  By hypothesis $\nu_+(p)\geq d-\qt+1$,
	so it follows that $\nu_-(p)\leq \qt-1$ and  therefore for any metric $h$ we have
	\[ \lambda_q^h(p)>0, \quad p \in M.\]
	
	For $0\leq r \leq \qt-1$ introduce the subsets
	\[V_r=\{p \in M: \nu_-(p)\leq r\}\cup F,\]
	so that $V_{\qt-1}=M$, and each $V_r$ is closed in $M$, since with respect to any metric  $h$  on $E$, we have $V_r=\{\lambda_{r+1}^h\geq 0\}\cup F$.
	We also let
	\[U_r=V_r \setminus V_{r-1}=\{p\in M: \nu_-(p)=r\}\]
	so that $V_r =\bigcup\limits_{j=0}^r U_j$. We will construct the metric whose existence is claimed 	
	using an inductive process.
	Starting with {a Hermitian metric $g_{r-1}$} on $E$
	in the $r$-th stage (where  {$1\leq r \leq \qt-1$}), we modify ${g}_{r-1}$ so that we obtain a metric $g_r$ for which  the form $S$ is strictly $\qt$-positive with respect to $g_r$ on the set $V_r$.

	To start the inductive process (i.e., $r=0$), we take $g_0$ to be the metric given in the hypotheses on the bundle $E$. Then on $V_0$,  since either all the eigenvalues are nonnegative and $\lambda_{\qt}^{g_0}>0$, or we are in the
	set $F$ where $S$ is strictly $\qt$-positive,
	it follows that the sum of the smallest $\qt$ Levi eigenvalues of $S$ with respect to $g_0$ is strictly positive, and therefore  $S$ is strictly $\qt$-positive with respect to $g_0$ on $V_0$.
	
	We now come to the $r$-th stage, where $1\leq r\leq  {\qt-1}$.  Therefore, there is already a metric $g_{r-1}$ on $E$ for which $S$ is strictly $\qt$-positive on $V_{r-1}$,   and on $U_r$, the eigenvalues with respect to $g_{r-1}$ satisfy
	$\lambda^{g_{r-1}}_r<0, \lambda^{{g}_{r-1}}_{r+1}\geq 0, \lambda^{{g}_{r-1}}_{\qt}>0 $.
	We claim that there is a smooth real-valued function $f\in \mathcal{C}^\infty(V_r)$ such that $f\geq 0$, $f$ vanishes in a neighborhood of $V_{r-1}$,
	and we have
	\begin{equation}
		\label{eq-f}{\displaystyle{\sum_{j=1}^{r} \lambda^{{g}_{r-1}}_j}} + (1+f)\cdot\left(\sum_{j=r+1}^{\qt} \lambda^{{g}_{r-1}}_j\right)>0
	\end{equation}
	at each point of $V_r$.
	
	By the induction hypothesis, on $V_{r-1}$,  we have  $\sum_{j=1}^{\qt} \lambda^{{g}_{r-1}}_j>0$, and therefore by continuity there is a neighborhood $\widetilde{V}_{r-1}$ of
	$V_{r-1}$ on $V_r$ where we continue to have $\sum_{j=1}^{\qt} \lambda^{{g}_{r-1}}_j>0$. Now consider the continuous function on $U_r=V_r \setminus V_{r-1}$ given by
	\[ \varphi =-\frac{\sum_{j=1}^{\qt} \lambda^{{g}_{r-1}}_j}{\sum_{j=r+1}^{\qt} \lambda^{{g}_{r-1}}_j}.\]
	The function $\varphi$ is well-defined since each of the summands in the denominator is nonnegative and $\lambda^{{g}_{r-1}}_{\qt}>0$ at each point of $M$. It follows further that $\varphi<0$ on $\widetilde{V}_{r-1}$, since both the numerator and denominator are strictly positive there. It therefore follows that there is a $\mathcal{C}^\infty$ function $f$ on $V_r$ such that $f\equiv 0$ in a neighborhood of $V_{r-1}$ and $f> \varphi$ on $U_r$. Such an $f$ clearly satisfies \eqref{eq-f}.
	
	Recall that $S^{g_{r-1}}_p$ denotes the Hermitian operator on $E_p$ which is Hermitian with respect to $g_{r-1}$ and satisfies $g_{r_1}(S^{g_{r-1}}_pu,v)=
	S_p(u,v)$ for $u,v\in E_p$.		For $p\in U_r$,  let $P_p^r$ denote the projection from the fiber  $E_p$  of $E$ onto the direct sum of the eigenspaces of $S^{g_{r-1}}_p$  of the negative eigenvalues $\lambda_1^{g_{r-1}}(p), \dots \lambda^{g_{r-1}}_r(p)$, where $P_p^r$ is orthogonal with respect
	to the metric $g_{r-1}$. We claim that $p\mapsto P_p^r$ is smooth on $U_r$.  More precisely, this map is a smooth section of the bundle $\mathrm{Hom}_{\mathbb{C}}(E,E)$. For any compact subset $K\subset U_r$, there are \emph{negative} numbers
	$\alpha, \beta$ such that $\alpha <\lambda_1^{g_{r-1}}(p)$ and $\lambda^{g_{r-1}}_r(p)< \beta$ for each $p\in K$. Let $G\subset \mathbb{C}$ denote the open disc with center at the point $\frac{1}{2}(\alpha+\beta)$ on the negative real axis and passing through the points $\alpha$ and $\beta$.  Then if $p\in K$, the negative eigenvalues
	$\lambda_1^{g_{r-1}}, \dots, \lambda_r^{g_{r-1}}$ of  $S_p$ with respect to $g_{r-1}$  lie in the open set $G$ and the other eigenvalues lie in $\mathbb{C}\setminus\overline{G}$, so  $S^{g_{r-1}}_p\in {\mathscr{W}_G}$, where
	$  {\mathscr{W}_G}$  is as in the statement of Proposition~\ref{prop-analytic}. Continuing to use the notation of Proposition~\ref{prop-analytic}, we have the representation
	\[ P_p^r= \pi_G(S^{g_{r-1}}_p)  \quad \text{ for } p\in K.\]
	From the smoothness of the map $p\mapsto S^{g_{r-1}}_p$ and the analyticity of the map $T\mapsto \pi_G(T)$, it follows that $p\mapsto P^r_p$ is smooth on $K$ and therefore on all of $U_r$.
	
	Now define a Hermitian  metric $g_r$ on $E|_{V_r}$ by setting for $p\in V_r$ and $X,Y\in E_p$

	\begin{equation}
		\label{eq-metricdef}
		g_r(X,Y) = \begin{cases}g_{r-1}(X,Y)+ f(p)\cdot g_{r-1}(P^r_pX,{P^r_p}Y)  & \text{if   } p\in U_r\\ g_{r-1}(X,Y) & \text{elsewhere on $V_r$}
		\end{cases}
	\end{equation}
	where $f\in \mathcal{C}^\infty(U_r)$ is as in \eqref{eq-f}. We can extend the smooth metric $g_r$ from $E|_{V_r}$ from the closed set $V_r$ to all of $M$ to obtain a Hermitian metric on $E$, which we continue to denote by $g_r$.
	
	To complete the induction, we will show that the sum of the smallest $\qt$ eigenvalues of $S^{g_r}$ is strictly positive on $V_r$.  Since on $V_{r-1}$ we have $g_{r}=g_{r-1}$ we only need to show this on $U_r$.
	Let $p\in U_r$. We have by the definitions of the operators $S^{g_{r-1}}$ and $S^{g_{r}}$ that
	\[ S_p(X,Y) =g_{r-1}(S^{g_{r-1}}_p(X), Y)=  g_r(S^{g_{r}}_p(X), Y).\]
Let
	$\{L_j(p)\}_{j=1}^{d}$ be a $g_{r-1}$-orthonormal set of  eigenvectors of the operator $S^{g_{r-1}}_p$ acting on $E_p$, corresponding to the eigenvalues  $\{\lambda_j^{g_{r-1}}(p)\}$, i.e., $S^{g_{r-1}}_pL_j(p)= \lambda_j^{g_{r-1}}(p)L_j(p)$. In general one cannot choose the sections $p\mapsto L_j(p)$  of $E$ over $U_r$ to be even continuous. Now the definition of the new metric $g_r$ in \eqref{eq-metricdef} shows that we have
	\[ g_r(L_j(p), L_k(p))= \begin{cases}
		(1+f(p))\delta_{jk} & \text{if both } j,k \leq r\\
		\delta_{jk} & \text{ if not}.
	\end{cases}\]
	Therefore for $p\in U_r$, the vectors
	\[ Z_j(p)=\begin{cases}
		\dfrac{1}{\sqrt{1+f(p)}}\cdot L_j(p) & \text{ if } 1\leq j \leq r \\
		L_j(p) & \text{ if } r+1 \leq j \leq d
	\end{cases}\]
	form an orthonormal basis of $E_p$ with respect to the new Hermitian metric $g_r$.  Now notice that
	\[g_r(S^{g_r}_p(Z_j(p)), Z_k(p))= \begin{cases}0 &\text{if  } j\not =k\\
		\dfrac{\lambda_j^{g_{r-1}}(p)}{1+f(p)} &\text{ if } j=k, \text{ and } 1\leq j \leq r\\
		\lambda_j^{g_{r-1}}(p)& \text{ if } j=k, \text{ and } r+1\leq j \leq d
	\end{cases}\]
	so that the basis $\{Z_j(p)\}$ diagonalizes $S^{g_r}_p$, and consequently, its eigenvalues with respect to the new metric $g_r$ at the point $p$ are (in ascending order)
	\[\dfrac{\lambda_1^{g_{r-1}}(p)}{1+f(p)}, \dots,\dfrac{\lambda_r^{g_{r-1}}(p)}{1+f(p)},\lambda_{r+1}^{g_{r-1}}(p), \dots, \lambda_{d}^{g_{r-1}}(p), \]
	the first $r$ of which are negative and the remaining are nonnegative, and the $\qt$-th one is known to be strictly positive. By the choice of the function $f$ in \eqref{eq-f}, it follows that the sum of the first $\qt$ of these
	eigenvalues is strictly positive, and this completes the proof.



\section{Proof of Theorem~\ref{thm-zq}}

  By hypothesis, $\partial\Omega$ satisfies condition $Z(q)$, so at each point of $\partial \Omega$ one of two mutually exclusive conditions hold: either that (i)  the Levi form has at least $(n-q)$-positive eigenvalues, or that (ii) the Levi form has at least $(q+1)$-negative eigenvalues. Noticing that the Levi form $\opL_\rho$ taken with respect to a smooth defining function $\rho$ is continuous on $\partial\Omega$, we see  each of the
conditions (i) and (ii) holds on an open set, so it follows that on each point of a connected component of $\partial\Omega$, exactly one of the conditions (i) or (ii) holds.

Fix a smooth defining function $\rho$ for $\Omega$.
For a boundary component $M_0$   where $\opL_\rho$ has at least $(n-q)$  positive eigenvalues, we apply Proposition \ref{prop-main} with
\[ d=n-1, \quad E=T^{1,0}M_0,\quad  S=\opL,\quad \qt=q, \quad F= \text{the empty set}. \]
We obtain a metric on this boundary component, with respect to which the Levi form is strictly $q$-positive.
For a boundary component $M_0$
where $\opL_\rho$ has at least $(q+1)$-negative eigenvalues,  $-\opL_\rho$ has at least $(q+1)$  positive eigenvalues.  We apply
Proposition \ref{prop-main} with
\[ d=n-1,\quad E=T^{1,0}M_0, \quad S=-\opL_\rho,\quad \qt=n-q-1, \quad F=  \text{the empty set}.  \]
Therefore, we obtain a metric on this boundary component with respect to which $-\opL_\rho$ is $(n-1-q)$-positive. Putting these together, we obtain a Hermitian
metric on $T^{1,0}(\bd\Om)$ and then extend it arbitrarily to the whole manifold $M$.
Theorem \ref{thm:L^2 theory, no weight} applies
and the proof of Theorem \ref{thm-zq} is complete.

\section{Constructing a metric for multiple Hermitian forms}
\subsection{A Result from Linear Algebra}

\begin{prop} Let $E$ be a finite dimensional inner product space, let $V$ and $W$ be subspaces of $E$, and let $\pi_V$ denote the orthogonal projection from $E$ to $V$. Then for each orthonormal basis $\{t_j\}$ of the space $W$ we have
	\[ \sum_{k=1}^{\dim W} \norm{\pi_V t_k}^2= \dim(V\cap W).\]
	\end{prop}

	\begin{proof}	
	 Consider the Hermitian operator $P$ on the inner product space $W$
	  (with the induced inner product)
	  characterized by the condition that
		\[ \ipr{Pz, w} = \ipr{\pi_V z, w	}, \quad z,w\in W. \]
			 We claim that $P$ coincides with the orthogonal projection from $W$ onto $V\cap W$, where  as before the inner product in the subspace $W$ is that induced from $E$.
	To see this, if $z\in V\cap W$, then for each $w\in W$ we have $\ipr{Pz, w}= \ipr{z,w}$, so $Pz=z$. On the other hand if $\zeta\in V^\perp \cap W$ then we have
	$\ipr{P\zeta, w}=\ipr{\pi_V \zeta, w}= 0$ for each $w\in W$, so $P\zeta=0$,  establishing the claim.

	We now compute the trace of $P$ (as an operator on the inner product space $W$) in two ways. Since $P$ is a projection from $W$ onto $V\cap W$, in an appropriate basis of $W$ the matrix of $P$ is diagonal, with $\dim(V\cap W)$ diagonal entries each equal to 1 and the rest equal to zero.	Therefore $ \tr(P)=\dim(V\cap W).$
	On the other hand with respect to the orthonormal basis $\{t_1,t_2,\dots, t_q\}$ of the space $W$, the matrix of the operator $P$ is $(a_{jk})_{j,k=1}^q$ where
 \[a_{jk}= \ipr{Pt_j, t_k}= \ipr{\pi_V t_j, t_k}= \ipr{\pi_V^2 t_j, t_k } = \ipr{\pi_V t_j, \pi_V t_k},\]
 where we have used the fact that $\pi_V$ is idempotent and self-adjoint.
 Therefore
 \begin{equation}\label{eq-trace2}
 	\tr(P)= \sum_{k=1}^q a_{kk}=  \sum_{k=1}^q \ipr{\pi_V t_k, \pi_V t_k}= \sum_{k=1}^q\norm{ {\pi_V t_k}}^2.
 \end{equation}
The result follows.

\end{proof}

\begin{cor}
	\label{cor-lowerbound}
Let $E$ be an $n$-dimensional inner-product space, let $1\leq q \leq n$ and let $V$ be an $n-q+1$ dimensional subspace of $E$. Denote by $\pi_V$ the orthogonal projection from $E$ onto $V$, and let $t_1,\dots, t_q$ be a collection of orthonormal vectors in $E$.  Then we have
\[ \sum_{k=1}^q \norm{\pi_V t_k}^2 \geq 1.\]
\end{cor}
\begin{proof} In the preceding proposition, take $W$ to be the span of the vectors $t_1, \dots, t_q$, so that  $\dim W=q$. 	 Notice now that since $\dim V=n-q+1,\dim W=q$ and $\dim E=n$, it follows that
	$\dim (V\cap W)\geq 1$. The result follows from the proposition.
\end{proof}
\subsection{The construction for a continuous subbundle}

We begin with the result that we will need in the proof of Theorem \ref{thm:continuously varying}, in the presence of a continuous subbundle.
\begin{thm} \label{thm-keyconstruction}
	Let $S$ be a smooth compact manifold, $E\to S$ a smooth complex vector bundle of rank $d$, and $V\to S$ a continuous subbundle of $E$ of rank $d-q+1$, where $1\leq q\leq d$.  Let $\{Q_j\}_{j=1}^N$ be a finite collection of continuous Hermitian forms on the bundle $E$ such that the restriction of each $Q_j$ to the subbundle $V$ is positive definite. Then there is a smooth Hermitian metric $g$ on $E$ such that for each $1\leq j \leq N$,  the form  $Q_j$  is strictly $q$-positive with respect to $g$.
\end{thm}

\begin{proof}
	It is sufficient to show the following:

	(*)  If $H$ is a continuous Hermitian form on $E$ which restricts to a positive definite form on $V$ and $\gamma$ is a continuous Hermitian metric on $E$, there is a
	constant $C(H,\gamma)>0$ such whenever $\kappa\geq C(H,\gamma)$, we have $H^h:E\to E$ is strictly $q$-positive, where the metric $h$ is given by
	\begin{equation}\label{eq-g}
		h(z,w)= \gamma(z,w)+ \kappa \cdot\gamma(P_{V^\perp}(z), P_{V^\perp} (w)),
	\end{equation}
	   where $P_{V^\perp}:E\to V^\perp$ is the orthogonal projection with respect to the metric $\gamma$ on the subbundle $V^\perp$ of $E$ whose fibers are orthogonal to those of $V$ with respect to $\gamma$. Notice that this $h$ is a continuous Hermitian metric on $E$.
	
	  Indeed, granted the statement (*), we can choose an arbitrary continuous Hermitian metric $\gamma$ on $E\to S$, and let $h$ be given by \eqref{eq-g}, where we take
	  \[ \kappa=\max\left\{C(Q_j,\gamma), 1\leq j \leq N\right\}.\]
        Then each $Q_j$ is strictly $q$-positive with respect to $h$. Let $g$ be a smooth Hermitian metric on $S$
    uniformly close to $h$. Such an approximation clearly exists locally, and using a partition of unity we can glue such local approximations  on the compact manifold $S$. Such a gluing preserves uniform closeness of the local approximations. Notice that if $h$ is sufficiently uniformly close to $g$ then $Q_j^g$ is also strictly $q$-positive for each $j$, and this is what we want.

	  To prove (*), let
	  \begin{equation}
	  \label{eq-a1def}
	   A_1 = \inf\{ H_p(z,z):p\in S, z\in V_p, \gamma_p(z,z)=1\}.
	  \end{equation}
	  We claim that $A_1>0$, and the infimum in \eqref{eq-a1def} is a minimum. For $p\in S$, by \eqref{eq-rayleigh} we have
	  \[ \min\{ H_p(z,z): z\in V_p, \gamma_p(z,z)=1\}= \lambda_{1}^\gamma(H_p|_{V_p}),\]
  the smallest eigenvalue with respect to $\gamma$ of the restriction of  the form $H_p$ to the subspace $V_p$. Since the form $H$, the metric $\gamma$ and the subbundle $V\subset E$ are continuous, and eigenvalues depend continuously on a matrix,  we easily conclude that the function $p\mapsto \lambda_{1}^\gamma(H_p|_{V_p})$ is continuous on $S$. Further since for each $p$, the restriction of $H_p$ to $V_p$ is positive definite, it follows that  this function is strictly positive at each point of $S$. Consequently, by the compactness of $S$, the number $A_1$, which is the infimum of this function
  is strictly positive and the infimum is a minimum.

  Now consider
  \begin{equation}
  	\label{eq-a2def}
  	A_2= \sup\left\{\abs{H_p(z,z)}: p\in S, z\in V_p^\perp, \gamma_p(z,z)=1\right\}.
  \end{equation}
  We claim $0\leq A_2 <\infty$.  Let $p\in S$. Then clearly we have
  \[  \sup\left\{\abs{H_p(z,z)}: z\in V_p^\perp, \gamma_p(z,z)=1\right\} = \max\left\{\abs{\lambda_1^\gamma(H_p|_{V^\perp})}, \abs{\lambda_{\max}^\gamma(H_p|_{V^\perp})}\right\},\]
  where $\lambda_{\max}^\gamma$ stands for the largest eigenvalue of the restricted Hermitian form with respect to the metric $\gamma$.
  	  An argument similar to above shows that as a function of $p$, the quantity above is continuous, and
  	  therefore it is bounded and consequently its supremum $A_2<\infty$.    Lastly, introduce the quantity
	  \begin{equation}
	  	\label{eq-a3def}
	  	A_3 = \sup \left\{\abs{H_p(z,w)}: p\in S, z\in V_p, w\in V_p^\perp, \gamma_p(z,z)=\gamma_p(w,w)=1\right\}.
	  \end{equation}
	  We claim that $0\leq A_3 <\infty$. Notice first that  if $p\in S$, then we have
	  \begin{align}
	    &\sup \left\{\abs{H_p(z,w)}: z\in V_p, w\in V_p^\perp, \gamma_p(z,z)=\gamma_p(w,w)=1\right\}\nonumber \\\leq &  \sup \left\{\abs{H_p(z,w)}: z\in E, w\in E, \gamma_p(z,z)=\gamma_p(w,w)=1\right\}\nonumber\\
	    = &\max\left\{\abs{\lambda_1^\gamma(H_p)}, \abs{\lambda_{\max}^\gamma(H_p)}\right\}.\label{eq-a3}
	  \end{align}
	  An argument similar to the above two cases shows that  the quantity in \eqref{eq-a3} is a continuous function of $p$. It follows that	  $A_3<\infty$.

	  It now follows that we can choose the number $C=C(H,\gamma)$ so large that we have
	  \begin{equation}
	  	\label{eq-a1bis}
	  	A_1 - \frac{q\cdot A_2}{1+C} - 2 \frac{q\cdot A_3}{\sqrt{1+C}}>0.
	  \end{equation}
	 We claim that  this value of $C=C(H,\gamma)$ works in (*).
	  To justify the claim,  let $\kappa\geq C$ and let $p\in S$, and we need to show that $H_p$ is a strictly $q$-positive operator on $E_p$ with respect to $h$, where
	  $h$ is as in \eqref{eq-g}. Notice that $h$ depends on  $\kappa$.

	  Let $t\in E_p$ be a vector of unit length with respect to the metric $h_p$.
	  Set $u=P_{V}t, v=P_{V^\perp t}$, where $P_V: E\to V$ is the bundle map whose section over $p\in S$ is the orthogonal projection with respect to $\gamma_p$ from $E_p$ to $V_p$.
	  Then we have
	  \begin{align*}
	  	1&=h_p(t,t)= \gamma_p(t,t)+ \kappa\cdot \gamma_p(P_{V^\perp} t, P_{V^\perp} t)\\
	  	& = \gamma_p(P_V t, P_Vt)+(1+\kappa)\gamma_p(P_{V^\perp} t, P_{V^\perp}t)\\
&	  = \gamma_p(u,u)+ (1+\kappa)\gamma_p(v,v),
	\end{align*}
	  from which we see that
	  \begin{equation}
	  	\label{eq-a2bis}
	  	\gamma_p(u,u)\leq 1, \quad \gamma_p(v,v) \leq \frac{1}{1+\kappa}\leq \frac{1}{1+C}.
	  \end{equation}

	  Therefore, since $t=u+v$,
	  \begin{align}
	  	&H(t,t)=H(u,u)+H(v,v)+ 2 \Re H(u,v)\notag\\
	  	&\geq H(u,u)- \abs{H(v,v)}- 2 \abs{H(u,v)}\notag\\
	  	&\geq A_1\gamma_p(u,u) -A_2\gamma_p(v,v) -2A_3 \sqrt{\gamma_p(u,u)}\sqrt{\gamma_p(v,v)}\notag\\
	  	&\geq A_1 \gamma_p(P_V t,P_Vt) - \frac{A_2}{1+C}-2 \frac{A_3}{\sqrt{1+C}}, \label{eq-a4bis}
	  \end{align}
	  where we have used the definitions of the $A_j$ as well as the bounds \eqref{eq-a2bis}.

	  To show that $H_p$ is strictly $q$-positive with respect to $h_p$, it suffices by Theorem~\ref{thm-schur} above to show that the restriction of the form   $H_p$ to each $q$-dimensional subspace of $(E_p,\gamma_p)$ has positive trace with respect to $h_p$. Let  $W$ be a $q$-dimensional subspace of $E_p$, and choose an $h_p$-orthonormal basis $t_1,\dots, t_q$ of $W$. Then we have
	  \begin{align*}
	  	\tr_h(H_p)&= \sum_{k=1}^q H_p(t_k, t_k)\notag\\
	  	&\geq  \sum_{k=1}^q\left(  A_1 \gamma(P_V t_k,P_Vt_k) - \frac{A_2}{1+C}-2 \frac{A_3}{\sqrt{1+C}}\right)&\text{by \eqref{eq-a4bis}}\\
	  	&= A_1  \sum_{k=1}^q \gamma(P_V t_k,P_Vt_k)- \frac{q\cdot A_2}{1+C}-2 \frac{q\cdot A_3}{\sqrt{1+C}} \\
	  	&\geq A_1 - \frac{q\cdot A_2}{1+C}-2 \frac{q\cdot A_3}{\sqrt{1+C}}&\text{ by Corollary~\ref{cor-lowerbound}}\\& >0.
	  \end{align*}

\end{proof}

\subsection{The construction in the top degree}

We next consider the construction that we will use in the $q=n-1$ case considered in Theorem \ref{thm:q_equals_n_minus_one}, in which we  still assume the existence of a direction in which both the Levi form and the Hessian of the weight function is positive, though we no longer assume that this direction varies continuously on the boundary. We will prove the following.
\begin{thm} \label{thm-keyconstruction_top_degree}
	Let $S$ be a smooth compact manifold and $E\to S$ a smooth complex vector bundle of rank $d$.  Let $Q_1$ and $Q_2$ be continuous Hermitian forms on the bundle $E$ such that at each point $p\in S$, there is a vector $v$ in
	the fiber of $E$ over the point $p$,  for which $Q_1(v,v)>0$ and $Q_2(v,v)>0$.
	Then there is a smooth Hermitian metric $g$ on $E$ such that  $\Tr_g Q_j>0$ for each $j\in\{1,2\}$.
\end{thm}

To prove this key result, we will need the following special case of Jacobi's formula for the derivative of the determinant (see \cite[p. 169]{magnus}, \cite[p. 798]{kline2}), obtained when the matrix under consideration is Hermitian.  Since this result is standard, we omit the proof.
\begin{lem}
\label{lem:determinant_derivative}
  For integers $d,m\geq 1$, let $\mathcal{O}\subset\mathbb{R}^m$ be an open set and let $\{M_x\}_{x\in\mathcal{O}}$ be a family of positive-definite $d\times d$ Hermitian matrices with entries in $C^1(\mathcal{O})$.  For any $1\leq j\leq m$, we have
  \begin{equation}
  \label{eq:determinant_derivative}
    \frac{\partial}{\partial x_j}\log\det M_x=\Tr\left((M_x)^{-1}\frac{\partial}{\partial x_j}M_x\right)
  \end{equation}
  on $\mathcal{O}$.
\end{lem}



The next Lemma contains the main part of the proof of Theorem~\ref{thm-keyconstruction_top_degree}.

\begin{lem}
\label{lem:two_Hermitian_forms}
	Let $E$ be a finite dimensional inner product space with $\dim E \geq 2$,
	and denote by $\mathscr{Q}$ the collection of ordered pairs
	of Hermitian forms on $E$ sharing a  common positive direction, i.e.,
$\mathscr{Q}$ consists of pairs $(Q_1, Q_2)$ where each $Q_j$ is a Hermitian form on $E$, and
	\[ \{v\in E: Q_1(v,v)>0\}\cap \{v\in E: Q_2(v,v)>0\} \not=\emptyset.\]
	Denoting by $\mathscr{H}$ the cone of Hermitian metrics on $E$, there is a continuous map
	\[ h:\mathscr{Q}\to \mathscr{H}\]
	such that the traces $\tr_{h(Q_1,Q_2)}Q_1$ and $\tr_{h(Q_1,Q_2)}Q_2$ computed with respect to the metric $h(Q_1,Q_2)$ are both positive.
\end{lem}

\begin{rem} The space of Hermitian forms on a finite dimensional vector space has a natural linear topology
	which can be represented by a variety of norms (see \cite[Section~5.6]{HoJo85})). This induces the topologies
	given on the spaces $\mathscr{Q}$ and $\mathscr{H}$.
\end{rem}

\begin{proof}
Let $d=\dim E$ and denote by $\ipr{\cdot, \cdot}$  its metric.  For $x\in \rl^2$, let $\ipr{\cdot, \cdot}_x$ denote the Hermitian form on $E$ given by
	\[ \ipr{\cdot, \cdot}_x=\left<\cdot,\cdot\right>-x_1Q_1(\cdot,\cdot)-x_2Q_2(\cdot,\cdot). \]
For notational simplicity,  when $ \ipr{\cdot, \cdot}_x$ is a metric,  for a Hermitian form $Q$ on $E$,  we will also use the shorthand  $\Tr_xQ$ for the trace of the form $Q$ with respect to the metric $ \ipr{\cdot, \cdot}_x$, instead of
 the more correct notation $\Tr_{ \ipr{\cdot, \cdot}_x}Q$.
	We will prove the lemma by showing that
	 there exists a point $\gamma(Q_1,Q_2)\in\mathbb{R}^2$ such that
	\begin{itemize}
		\item $\gamma_1(Q_1,Q_2)>0$ and $\gamma_2(Q_1,Q_2)>0$,
		\item the form $	\left<\cdot,\cdot\right>_{\gamma(Q_1,Q_2)}$ is a
	 Hermitian metric on $\mathbb{C}^d$, and
		\item the traces $\Tr_{\gamma(Q_1,Q_2)} Q_1$ and $\Tr_{\gamma(Q_1,Q_2)} Q_2$ of $Q_1$ and $Q_2$ computed with respect to this new metric are both positive.
	\end{itemize}
	Furthermore, $\gamma(Q_1,Q_2)$ can be chosen to depend continuously on $Q_1$ and $Q_2$, so in the conclusion of the lemma we may take $h(Q_1,Q_2)=
	\ipr{\cdot, \cdot}_{\gamma(Q_1,Q_2)}$.

	 Let $\mathcal{O}\subset\mathbb{R}^2$ denote the set of all $x\in\mathbb{R}^2$ such that $\left<\cdot,\cdot\right>_x$ is positive definite.  Clearly $\mathcal{O}$ is an open set and $(0,0)\in\mathcal{O}$.  For any $x\in\mathbb{R}^2\backslash\{(0,0)\}$, let $\lambda_x$ denote the largest eigenvalue of $x_1 Q_1+x_2 Q_2$.  If $\lambda_x\leq 0$, then $tx\in\mathcal{O}$ for all $t\geq 0$.  If $\lambda_x>0$, then $tx\in\mathcal{O}$ for $t\geq 0$ if and only if $0\leq t<\lambda_x^{-1}$.  Hence, $\mathcal{O}$ is a starlike domain centered at $(0,0)$.  Let $\mathcal{O}^+$ denote the set of all $x\in\mathcal{O}$ such that $x_1>0$ and $x_2>0$.  When $x\in\mathcal{O}^+$, then $x_1Q_1(v,v)+x_2Q_2(v,v)>0$, so $\lambda_x>0$ and hence
	\begin{equation}
		\label{eq:bounded_in_first_quadrant}
		\mathcal{O}^+\text{ is a bounded subset of }\mathcal{O}.
	\end{equation}
	Fix an orthonormal basis $\{e_j\}_{1\leq j\leq d}$ for $E$ with respect to the base metric $\left<\cdot,\cdot\right>$.  For $x\in\mathcal{O}$, define a $d\times d$ positive-definite Hermitian matrix $g(x)$ by $g(x)=\left(g_{j\bar k}(x)\right)_{1\leq j,k\leq d}$ where $g_{j\bar k}(x)=\left<e_j,e_k\right>_x$.  Denote the inverse of this matrix by $g^{-1}(x)=\left(g^{\bar k j}(x)\right)_{1\leq j,k\leq d}$.  Since $g(x)g^{-1}(x)=I$ for all $x\in\mathcal{O}$, we may differentiate by $x_r$ for any $r\in\{1,2\}$ to obtain
	\[
	\left(\frac{\partial}{\partial x_r}g(x)\right)g^{-1}(x)+g(x)\frac{\partial}{\partial x_r}g^{-1}(x)=0,
	\]
	or
	\begin{equation}
		\label{eq:g_inverse_derivative}
		\frac{\partial}{\partial x_r}g^{-1}(x)=-g^{-1}(x)\left(\frac{\partial}{\partial x_r}g(x)\right)g^{-1}(x).
	\end{equation}
	Let $\{u_j(x)\}_{1\leq j\leq d}$ be an orthonormal basis for $\mathbb{C}^d$ with respect to $\left<\cdot,\cdot\right>_x$.  Then for every $1\leq j\leq d$, $e_j=\sum_{k=1}^d A_j^k(x) u_k(x)$ for some non-singular matrix $A(x)=(A_j^k(x))_{1\leq j,k\leq d}$.  For $1\leq j,k\leq d$,
	\[
	g_{j\bar k}(x)=\left<e_j,e_k\right>_x=\sum_{\ell=1}^d A_j^\ell(x)\overline{A_k^\ell(x)},
	\]
	so $g(x)=A(x)A^*(x)$.  If we right-multiply by $g^{-1}(x)A(x)$, we obtain
	\[
	A(x)=A(x)A^*(x)g^{-1}(x)A(x).
	\]
	Since $A(x)$ is non-singular, this necessarily implies that
	\begin{equation}
		\label{eq:coordinate_change_metric}
		A^*(x)g^{-1}(x)A(x)=I\text{ for all }x\in\mathcal{O}.
	\end{equation}

	We define a function $\xi\in C^\infty(\mathcal{O})$ by $\xi(x)=-\log\det g(x)$. This function $\xi$ is independent of the choice of basis $\{e_j\}$ used to define the matrix $g$, since if $U$ is a unitary (with respect to $\ipr{\cdot,\cdot}$) change of coordinates the matrix $g$ transforms to $U^*gU$ which has the same determinant.  Observe that $g(0,0)$ is the identity matrix, so $\xi(0,0)=0$.  As $x\rightarrow b\mathcal{O}$, $\xi(x)\rightarrow\infty$, so \eqref{eq:bounded_in_first_quadrant} implies that
	\begin{equation}
		\label{eq:varphi_range}
		(0,\infty)\subset\range\left(\xi|_{\text{positive }x_1\text{-axis}}\right)\text{ and }(0,\infty)\subset\range\left(\xi|_{\text{positive }x_2\text{-axis}}\right).
	\end{equation}
	Fix $r\in\{1,2\}$.  Using \eqref{eq:determinant_derivative}, we have
	\[
	\frac{\partial\xi}{\partial x_r}(x)=-\Tr\left(g^{-1}(x)\frac{\partial}{\partial x_r}g(x)\right),
	\]
	where the trace is computed with respect to our base metric.  Since $g(x)$ is a linear function, we may easily compute the derivatives to obtain
	\begin{equation}
		\label{eq:varphi_derivative_base_coordinates}
		\frac{\partial\xi}{\partial x_r}(x)=\sum_{j,k=1}^ng^{\bar k j}(x)Q_r(e_j,e_k).
	\end{equation}
	Using \eqref{eq:coordinate_change_metric} to simplify after changing coordinates, we have
	\[
	\frac{\partial\xi}{\partial x_r}(x)=\sum_{j,k,\ell,m=1}^d\overline{A_k^m(x)}g^{\bar k j}(x)A_j^\ell(x)Q_r(u_\ell(x),u_m(x))=\sum_{j=1}^d Q_r(u_j(x),u_j(x)).
	\]
	Since we denote the trace with respect to $\left<\cdot,\cdot\right>_x$ by $\Tr_x$, we have
	\begin{equation}
		\label{eq:varphi_derivative}
		\frac{\partial\xi}{\partial x_r}(x)=\Tr_x Q_r\text{ for all }r\in\{1,2\}.
	\end{equation}
	For $r,s\in\{1,2\}$, we may use \eqref{eq:g_inverse_derivative} to differentiate \eqref{eq:varphi_derivative_base_coordinates} and obtain
	\[
	\frac{\partial^2\xi}{\partial x_s\partial x_r}(x)=-\sum_{j,k,\ell,m=1}^d g^{\bar k m}(x)\left(\frac{\partial}{\partial x_s}g_{m\bar\ell}(x)\right)g^{\bar\ell j}(x)Q_r(e_j,e_k).
	\]
	Once again, it is easy to compute derivatives of the linear functions comprising $g(x)$, so we have
	\[
	\frac{\partial^2\xi}{\partial x_s\partial x_r}(x)=\sum_{j,k,\ell,m=1}^d g^{\bar k m}(x)Q_s(e_m,e_\ell)g^{\bar\ell j}(x)Q_r(e_j,e_k).
	\]
	As before, we may use \eqref{eq:coordinate_change_metric} to simplify after changing coordinates and obtain
	\begin{align*}
		\frac{\partial^2\xi}{\partial x_s\partial x_r}(x)&=\sum_{j,k=1}^d Q_s(u_k(x),u_j(x))Q_r(u_j(x),u_k(x))\\
		&=\sum_{j,k=1}^d Q_r(u_j(x),u_k(x))\overline{Q_s(u_j(x),u_k(x))}.
	\end{align*}
	In other words, $\frac{\partial^2\xi}{\partial x_s\partial x_r}(x)$ is simply the inner product of $Q_r$ and $Q_s$ using the standard inner product for matrices with respect to the metric $\left<\cdot,\cdot\right>_x$ (see Section 5.6 in \cite{HoJo85} for further discussion of the associated $\ell^2$ norm).  Since $Q_1(v,v)>0$ and $Q_2(v,v)>0$ by hypothesis, neither Hermitian form vanishes, so we must have $\frac{\partial^2\xi}{\partial x_1^2}(x)>0$ and $\frac{\partial^2\xi}{\partial x_2^2}(x)>0$.  By the Cauchy-Schwarz inequality, we have $\frac{\partial^2\xi}{\partial x_1^2}(x)\frac{\partial^2\xi}{\partial x_2^2}(x)\geq\left(\frac{\partial^2\xi}{\partial x_1\partial x_2}(x)\right)^2$, so the Hessian of $\xi$ is positive semi-definite, and hence $\xi$ is convex.  Furthermore, $\xi$ is strictly convex unless $Q_1=\mu Q_2$ for some $\mu\in\mathbb{R}$.  Since $Q_1(v,v)>0$ and $Q_2(v,v)>0$, we must have $\mu>0$ in this case.  To summarize:
	\begin{gather}
		\label{eq:varphi_convexity}\xi\text{ is a convex exhaustion function on }\mathcal{O}\text{ and}\\
		\label{eq:varphi_strict_convexity}\text{either }\xi\text{ is strictly convex on }\mathcal{O}\text{ or }
		Q_1=\mu Q_2\text{ for some }\mu>0.
	\end{gather}
	We observe that \eqref{eq:varphi_convexity} implies that $\mathcal{O}$ is a convex set.
	
	If $Q_1=\mu Q_2$ for some $\mu>0$, then $\xi(x)=f(\mu x_1+x_2)$ for some real-valued function $f$ defined on some interval in $\mathbb{R}$. By \eqref{eq:varphi_convexity}, $f$ must be convex (we can easily check that $f$ is strictly convex, but we will not need this).  By \eqref{eq:bounded_in_first_quadrant}, there must exist a finite real number $b>0$ such that $\lim_{t\rightarrow b^-}f(t)=\infty$.  Since $f$ is convex, $f'$ is non-decreasing, and hence there exists $a\geq 0$ such that $f'(t)>0$ for all $a<t<b$, and, if $a>0$, $f'(t)\leq 0$ for all $0\leq t<a$.  By \eqref{eq:varphi_derivative}, $\Tr_x Q_1=\mu f'(\mu x_1+x_2)$ and $\Tr_x Q_2=f'(\mu x_1+x_2)$, so $\Tr_x Q_1>0$ and $\Tr_x Q_2>0$ for $x\in\mathcal{O}$ if and only if $a<\mu x_1+x_2<b$.  By \eqref{eq:varphi_range}, $f(a)\leq 0$, so there exists $a<c<b$ such that $f(c)=1$.  Let $\gamma(Q_1,Q_2)=\left(\frac{c}{2\mu},\frac{c}{2}\right)$, which is the midpoint of the line segment connecting $(0,c)$ to $\left(\frac{c}{\mu},0\right)$, a subset of the level curve $\xi^{-1}(\{c\})$.  Henceforth, \eqref{eq:varphi_strict_convexity} allows us to assume that $\xi$ is strictly convex.

	Consider the map $\psi:\mathcal{O}\rightarrow\mathbb{R}^2$ defined by $\psi(x)=\nabla\xi(x)$ for all $x\in\mathcal{O}$.  Suppose that $\psi(y)=\psi(z)$ for some $y,z\in\mathcal{O}$.  Since $\mathcal{O}$, the line segment $ty+(1-t)z\in\mathcal{O}$ for all $0\leq t\leq 1$.  By Rolle's Theorem, there must exist $0<s<1$ such that $\frac{d}{dt}\left((y-z)\cdot\nabla\xi(ty+(1-t)z)\right)\big|_{t=s}=0$.  Equivalently,
	\[
	\sum_{j,k=1}^2(y_j-z_j)\frac{\partial^2\xi}{\partial x_j\partial x_k}(sy+(1-s)z)(y_k-z_k)=0.
	\]
	Since $\xi$ is strictly convex, $\nabla^2\xi$ is positive definite, so this means that $y=z$.  Hence, $\psi$ is injective.  The Jacobian of $\psi$ is the Hessian of $\xi$, which is non-singular since $\xi$ is strictly convex.  Hence, the Inverse Function Theorem guarantees that $\psi$ admits a smooth local inverse.  Taken together, we see that $\psi$ is a smooth diffeomorphism from $\mathcal{O}$ onto its range.

	Let $\Gamma=\xi^{-1}(\{1\})$ be a level curve of $\xi$, and let $\tilde\Gamma$ denote the set of all $x\in\Gamma\cap\mathcal{O}^+$ such that $\psi_1(x)>0$ and $\psi_2(x)>0$.  Since $\xi$ is strictly convex, $\Gamma$ is the boundary of the strictly convex domain $\{x\in\mathcal{O}:\xi(x)<1\}$.  Hence, $\psi(\Gamma)$ is a smooth curve which crosses each radial line segment at most once, which implies that $\tilde\Gamma$ is a connected curve (if it is non-empty).  Now \eqref{eq:varphi_range} guarantees that there must exist $a>0$ and $b>0$ such that $\xi(a,0)=1$ and $\xi(0,b)=1$.  If there are two such values for either parameter, we choose the largest value, so that $\frac{\partial\xi}{\partial x_1}(a,0)>0$ and $\frac{\partial\xi}{\partial x_2}(0,b)>0$.  Hence $(a,0)$ and $(0,b)$ are endpoints of a smooth curve in $\mathcal{O}^+$.  By the Extended Mean Value Theorem, there must exist a point $y\in\Gamma\cap\mathcal{O}^+$ such that $(a,-b)$ is tangential to $\Gamma$ at $y$, and hence $\nabla\xi(y)$ is a positive multiple of $(b,a)$.  This means that $\frac{\partial\xi}{\partial x_1}(y)>0$ and $\frac{\partial\xi}{\partial x_2}(y)>0$, so \eqref{eq:varphi_derivative} implies that $\Tr_y Q_1>0$ and $\Tr_y Q_2>0$.  As a result, $y\in\tilde\Gamma$ and so $\tilde\Gamma$ is not empty.  Let $\gamma(Q_1,Q_2)$ denote the midpoint of $\tilde\Gamma$ with respect to the standard Euclidean arclength (a level curve of a convex function in a bounded set $\mathcal{O}^+$ must be of finite length).

	To see that $\gamma(Q_1,Q_2)$ depends continuously on $Q_1$ and $Q_2$, it suffices to note that $g(x)$ is a linear function of the entries of $Q_1$ and $Q_2$ with respect to the coordinates $\{e_j\}_{1\leq j\leq d}$, and hence $g(x)$ depends continuously on $Q_1$ and $Q_2$.  This means that $\xi(x)$ depends continuously on $Q_1$ and $Q_2$, and by \eqref{eq:varphi_derivative}, $\nabla\xi(x)$ also depends continuously on $Q_1$ and $Q_2$.  Since $1$ is never a critical value of $\xi$, $\Gamma\cap\mathcal{O}^+$ will depend continuously on $Q_1$ and $Q_2$ with respect to Hausdorff distance.  Since $\nabla\xi(x)$ depends continuously on $Q_1$ and $Q_2$, $\tilde\Gamma$ must also depend continuously on $Q_1$ and $Q_2$.  Since $\gamma(Q_1,Q_2)$ is defined to be the midpoint of $\tilde\Gamma$ (even in the case in which $Q_1=\mu Q_2$), $\gamma$ is a continuous function.

\end{proof}
Finally, we are ready to prove the analog to Theorem \ref{thm-keyconstruction} in the top degree when the subbundle is not continuous.

\begin{proof}[Proof of Theorem~\ref{thm-keyconstruction_top_degree}] Fix a Hermitian metric on the bundle $E$.  For each $p\in S$, now we have an inner product on $E_p$, and the two given Hermitian forms $Q_1(p), Q_2(p)$ for which by hypothesis there is a common positive direction. Therefore,
by Lemma~\ref{lem:two_Hermitian_forms} we have a Hermitian metric on $E_p$ given by 	$h(Q_1(p), h(Q_2(p)))$. If we define a continuous metric on $g_1$ on $E$ by
$g_1(p)=h(Q_1(p), h(Q_2(p)))$, then  $\tr_{g_1(p)}Q_j(p)$ is positive for each $p$ and $j\in \{1,2\}$. After a standard regularization of $g_1$, we obtain a smooth metric $g$ with the required properties (positive traces on compact sets are stable under perturbations).
\end{proof}

\section{Proof of Theorems~\ref{thm:continuously varying} and \ref{thm:q_equals_n_minus_one}}

\subsection{Proof of Theorem~\ref{thm:continuously varying}}

Let $\rho$ be a defining function for the smoothly bounded domain $\Om\subset M$. This gives rise to a Levi form $\opL_\rho$ which is a Hermitian form
on the complex vector bundle $T^{1,0}(\partial\Om)$ (see \eqref{eq-lrho}). Notice that $T^{1,0}(\partial\Om)$ has rank $(n-1)$ on the smooth manifold
$\partial\Omega$ and 	$\opL_\rho$ is  a Hermitian form on  it with at least $(n-q)$ positive eigenvalues or at least $q+1$ negative eigenvalues, depending on the case.

 Denote by $\mathcal{H}_\varphi$ the restriction of the complex Hessian of $\varphi$  to the boundary, i.e.\violet{,} for $p\in \partial \Om$, $X,Y\in T^{1,0}_p(\partial \Omega)$, we have
\[ \mathcal{H}_\varphi(X,Y)=\partial \dbar \varphi (X,\ol{Y}). \]
Then $\mathcal{H}_\varphi$ is  a Hermitian form on the bundle $T^{1,0}(\partial\Om)$. Since the boundary and the weight $\varphi$ are smooth,
these Hermitian forms $\mathcal{H}_\varphi$ and $\opL_\rho$ are  smooth.

We first consider the case in which $\opL_\rho$ and $\mathcal{H}_\varphi$ are positive definite on a subbundle of rank $n-q$.  We apply Theorem~\ref{thm-keyconstruction} with
$S=\bd\Om$, $E=T^{1,0}(\bd\Om)$ (so that $d=n-1$) and
let $V$ be the subbundle of $E$ of rank $d-q+1=n-q$ on
which both the Hermitian forms $L_\rho$ and $\mathcal{H}_\varphi$ are positive. Then there is a smooth metric $h$ on the bundle $T^{1,0}(\bd\Om)$ of rank $(n-1)$ such that both $L_\rho$ and $\mathcal{H}_\varphi|_{\bd\Om}$ are strictly $q$-positive with respect to $h$.

By the continuity of eigenvalues as a function of the matrix, and the compactness of $\partial \Omega$, there is a $\delta_0>0$ such that $\mathcal{H}_{\varphi+\delta\rho} = \mathcal{H}_\varphi+ \delta \mathcal{H}_\rho$  has $(n-q+1)$ positive eigenvalues as a Hermitian form on $T^{1,0}M|_{\partial \Omega}$ whenever $0\leq\delta\leq\delta_0$. It follows that there is a neighborhood $\mathcal{U}$ of $\partial \Omega$ in $M$ where $\mathcal{H}_{\varphi+\delta\rho}$  continues to have $(n-q+1)$ positive eigenvalues whenever $0\leq\delta\leq\delta_0$. Since the signature of the complex Hessian is independent of the choice of metric, $\delta_0$ is also independent of the choice of metric.  Choose $\epsilon_0>0$ so small that $\{-\epsilon_0\leq \rho\leq 0\}\subset \ol{\Omega}\cap \mathcal{U}$.   Let $\chi:\mathbb{R}\rightarrow\mathbb{R}$ be a smooth, convex, non-decreasing function such that $\chi(t)=0$ whenever $t\leq -1$ and $\chi'(0)=1$.  For $0<\epsilon< \epsilon_0$, set \[\varphi_\epsilon=\varphi+\delta_0\epsilon\chi(\epsilon^{-1}\rho).\]
Let $g_0$ be a Hermitian  metric on $T^{1,0}M|_{\partial\Om}$ extending the metric $h$ constructed above
on $T^{1,0}(\partial\Om)$.  We claim the following:
\begin{enumerate}
	\item $\mathcal{H}_{\varphi_\epsilon}$ has $(n-q+1)$ positive eigenvalues in $\ol{\Omega}$ as a Hermitian form on $T^{1,0}M$. \label{item-1}
	\item $\mathcal{H}_{\varphi_\epsilon}$  is strictly $q$-positive as a Hermitian form on $T^{1,0}(\partial\Om)$ with respect to $g_0$. \label{item-2}
	\item For sufficiently small $\epsilon$, the form  $\mathcal{H}_{\varphi_\epsilon}$  is strictly $q$-positive as a Hermitian form on $T^{1,0}M|_{\partial\Om}$ with respect to $g_0$. \label{item-3}
\end{enumerate}
To see \eqref{item-1}, notice first that $\varphi_\epsilon=\varphi$ on the set $\{\rho(z)\leq-\epsilon\}$ which is contained in $ \ol{\Omega}\setminus \mathcal{U}$,
provided $\epsilon< \epsilon_0$. By a  computation:
\begin{equation}\label{eq-hessian1} \mathcal{H}_{\varphi_\epsilon}=  \mathcal{H}_{\varphi}+\delta_0\chi'(\epsilon^{-1}\rho)\mathcal{H}_\rho+ \delta_0\epsilon^{-1}\chi''(\epsilon^{-1}\rho)\partial\rho\wedge\dbar\rho.
\end{equation}
Since the third term of this sum will introduce a positive semi-definite term to $\mathcal{H}_{\varphi_\epsilon}$ and since $0\leq\delta_0\chi'(\epsilon^{-1}\rho)\leq\delta_0$ on $\mathcal{U}\cap\overline\Omega$, we see that $\mathcal{H}_{\varphi_\epsilon}$ must have at least $n-q+1$ positive eigenvalues on $\mathcal{U}$.

To see \eqref{item-2}, note first that
\[ \mathcal{H}_{\varphi_\epsilon}|_{T^{1,0}(\partial\Om)}=  \mathcal{H}_{\varphi}|_{T^{1,0}(\partial\Om)}+\delta_0\opL_\rho.\]
Now since the metric $g_0$ restricts to the metric $h$ constructed above on $T^{1,0}(\bd\Om)$ and each
of the Hermitian forms $\mathcal{H}_{\varphi}$ and $\opL_\rho$ is strictly $q$-positive with respect to $h$, it
follows that $ \mathcal{H}_{\varphi_\epsilon}$ is also  strictly $q$-positive with respect to $h$ (and therefore $g_0$) on $T^{1,0}(\partial\Om)$. Notice that the fact that the sum of two strictly $q$-positive Hermitian forms is again strictly $q$-positive is an immediate consequence of  Theorem~\ref{thm-schur}.

To see \eqref{item-3}, note that \eqref{eq-hessian1} becomes in this case
\[  \mathcal{H}_{\varphi_\epsilon}|_{T^{1,0}M|_{\bd\Om}}=  \mathcal{H}_{\varphi}|_{T^{1,0}M|_{\bd\Om}}+\delta_0\mathcal{H}_\rho|_{T^{1,0}M|_{\bd\Om}}+ \delta_0\epsilon^{-1}\chi''(0)\partial\rho\wedge\dbar\rho|_{T^{1,0}M|_{\bd\Om}}.\]
We will use the characterization of Theorem~\ref{thm-schur} to show that  $ \mathcal{H}_{\varphi_\epsilon}$
is strictly $q$-positive on $T^{1,0}M|_{\bd\Om}$.  Let
\[ B_1= \max\{ \tr_{g_0}(\mathcal{H}_{\varphi}(z)|_U): z\in \partial \Omega, U \subset T^{1,0}_zM \text{ is $q$-dimensional}\}\]
and
\[ B_2= \max\{ \tr_{g_0}(\mathcal{H}_{\rho}(z)|_U): z\in \partial \Omega, U \subset T^{1,0}_zM \text{ is $q$-dimensional}\}.\]
The existence of these finite maxima follows by standard continuity and compactness arguments.
 Let $p\in \bd\Omega$ and $W$ be a $q$-dimensional
subspace of $T^{1,0}_pM$, and let $t_1,\dots, t_q$ be a $g_0$-orthonormal basis of $W$. Then we have
	\begin{align}
	\tr_{g_0}(\mathcal{H}_{\varphi_\epsilon}(p)|_W)&= \sum_{j=1}^q \mathcal{H}_{\varphi_\epsilon}(t_j, t_j)\nonumber\\
	&= \tr_{g_0}(\mathcal{H}_{\varphi}(p)|_W) + \delta_0 \tr_{g_0}(\mathcal{H}_\rho(p)|_W)+ \frac{\delta_0}{\epsilon}\chi''(0)\sum_{j=1}^q \abs{\partial \rho(t_j)}^2.
\label{eq-trace1}
\end{align}
In the case $\sum_{j=1}^q \abs{\partial \rho(t_j)}^2=0$, we have  $W\subset T^{1,0}_p(\partial \Omega)$ and
\[ \tr_{g_0}(\mathcal{H}_{\varphi_\epsilon}(p)|_W)= \tr_{h}(\mathcal{H}_{\varphi}(p)|_W) + \delta_0 \tr_{h}(\mathcal{L}_{\rho}(p)|_W)>0, \]
since both $\mathcal{H}_\varphi$ and the Levi form $\opL_\rho$ are strictly $q$-positive on $T^{1,0}(\bd\Om)$ with respect to $h$, and $g_0$ extends $h$.
By continuity there exists an $\eta>0$ independent of $\epsilon>0$ such that
\[
  \tr_{g_0}(\mathcal{H}_{\varphi}(p)|_W) + \delta_0 \tr_{g_0}(\mathcal{H}_\rho(p)|_W)>0
\]
whenever $\sum_{j=1}^q \abs{\partial \rho(t_j)}^2<\eta$.  Since the third term in \eqref{eq-trace1} is non-negative, we have $\tr_{g_0}(\mathcal{H}_{\varphi_\epsilon}(p)|_W)>0$ if $\sum_{j=1}^q \abs{\partial \rho(t_j)}^2<\eta$. 
In the case $\sum_{j=1}^q \abs{\partial \rho(t_j)}^2\geq \eta$ (and $W \subset T^{1,0}(M)|_{\p\Om}$ 
is no longer a subset of $T^{1,0}(\p\Om)$), we take
\[ \epsilon< \min\left(\frac{\delta_0\chi''(0)\eta}{B_1+\delta_0 B_2}, \epsilon_0\right)\]
and observe that
	\begin{align*}
			\tr_{g_0}(\mathcal{H}_{\varphi_\epsilon}(p)|_W)	& \geq -B_1-\delta_0 B_2 + \frac{\delta_0}{\epsilon}\chi''(0)\sum_{j=1}^q \abs{\partial \rho(t_j)}^2\\
			&\geq -B_1-\delta_0 B_2 + \frac{\delta_0}{\epsilon}\chi''(0)\eta\\
			& > 0.
		\end{align*}

In Proposition~\ref{prop-main}, we take the bundle
$E$ to be $T^{1,0}M$
(so $d=n$) and for the closed set $F\subset M$ we take $F=\partial\Omega$.
We take $g_0$ to be an extension of $h$ as above, and $\epsilon$ small enough
so that the function $\varphi_\epsilon$ satisfies the three conditions above, and let this be the function $\phi$ of Proposition~\ref{prop-main}. Extend $g_0$ smoothly
to a neighborhood of $\bd\Om$.
Then the hypotheses of Proposition~\ref{prop-main} are satisfied for $\tilde q=q$, and consequently there is a metric $g$ on a neighborhood of $\ol{\Om}$ which coincides with $g_0$ near $\partial\Omega$, and with respect to which $\mathcal{H}_{\varphi_\epsilon}$ is strictly $q$-positive. Therefore,
by Theorem~\ref{thm:L^2 theory, with weights} we have
that $H^{p,q}_{L^2}(\Omega)=0$.

Now we consider the case in which $\opL_\rho$ and $\mathcal{H}_\varphi$ are negative definite on a subbundle of rank $q+1$.  We again apply Theorem~\ref{thm-keyconstruction} with
$S=\bd\Om$ and $E=T^{1,0}(\bd\Om)$, but this time $V$ is the subbundle of $E$ of rank $d-(n-q-1)+1=q+1$ on
which both the Hermitian forms $-L_\rho$ and $-\mathcal{H}_\varphi$ are positive. Then there is a smooth metric $h$ on the bundle $T^{1,0}(\bd\Om)$ of rank $(n-1)$ such that both $-L_\rho$ and $-\mathcal{H}_\varphi|_{\bd\Om}$ are strictly $(n-q-1)$-positive with respect to $h$.

Let $\tilde L_n\in T^{1,0}(M)$ be a smooth vector field such that on $\partial\Omega$, $\tilde L_n$ is non-trivial and transverse to $T^{1,0}(\partial\Omega)$.  At any $p\in\partial\Omega$, let $\{L_j\}_{1\leq j\leq n-1}$ be an orthonormal basis for $T^{1,0}_p(\partial\Omega)$ with respect to $h$.  Since the restriction of the complex Hessian of $\varphi$ to $T^{1,0}(\partial\Omega)$ is nondegenerate, the matrix $\left(\ddbar\varphi(L_j,\bar L_k)\right)_{1\leq j,k\leq n-1}$ is invertible, and hence there exists a unique vector $v\in\mathbb{C}^{n-1}$ such that $\sum_{j=1}^{n-1} v^j\ddbar\varphi(L_j,\bar L_k)=-\ddbar\varphi(\tilde L_n,\bar L_k)$ for all $1\leq k\leq n-1$.  By the Implicit Function Theorem, the dependence of $v$ on $p$ is smooth.  For any $\epsilon>0$, we set $L_n^\epsilon=\epsilon\left(\tilde L_n+\sum_{j=1}^{n-1}v^j L_j\right)$, so that $L_n^\epsilon$ is smooth on $\partial\Omega$ and $\ddbar\varphi(L_n^\epsilon,\bar L)=0$ for all $L\in T^{1,0}(\partial\Omega)$.  There is a unique Hermitian metric $g_0^\epsilon$ on $T^{1,0}M|_{\partial\Om}$ such that $\{L_j\}_{1\leq j\leq n-1}\cup\{L_n^\epsilon\}$ is orthonormal on $\partial\Omega$.  Observe that $g_0^\epsilon$ so defined must equal $h$ on $T^{1,0}(\partial\Om)$.

Let $\{\mu_j\}_{1\leq j\leq n-1}$ denote the eigenvalues with respect to $h$ of the complex Hessian of $\varphi$ when restricted to $T^{1,0}(\partial\Omega)$, arranged in non-decreasing order.  By construction of $g_0^\epsilon$, these are also eigenvalues of the full complex Hessian of $\varphi$ with respect to $g_0^\epsilon$. Furthermore, $L_n^\epsilon$ is an eigenvector of the complex Hessian of $\varphi$ with respect to $g_0^\epsilon$ with eigenvalue
\[
  \mu_n^\epsilon=\ddbar\varphi(L_n^\epsilon,\bar L_n^\epsilon)=\epsilon^2\ddbar\varphi\left(\tilde L_n+\sum_{j=1}^{n-1}v^j L_j,\overline{\tilde L_n}+\sum_{j=1}^{n-1}\bar v^j\bar L_j\right).
\]
Observe that $\mu_n^\epsilon\rightarrow 0$ as $\epsilon\rightarrow 0^+$.  Since the complex Hessian of $-\varphi$ is strictly $(n-q-1)$-positive on $T^{1,0}(\partial\Omega)$ with respect to $h$, we have
\begin{equation}
\label{eq:n-q-1_positivity}
  \sum_{j=q+1}^{n-1}(-\mu_j)>0\text{ on }\partial\Omega.
\end{equation}
Since the eigenvalues are arranged in non-decreasing order, this is only possible if $\mu_{q+1}<0$, and hence we must also have $\mu_q<0$.  Fix $\epsilon>0$ sufficiently small so that
\begin{equation}
\label{eq:epsilon_characterized}
  \sum_{j=q+1}^{n-1}(-\mu_j)>\mu_n^\epsilon>\mu_q\text{ on }\partial\Omega.
\end{equation}
Since $\epsilon$ is now fixed, we may write $g_0=g_0^\epsilon$, $L_n=L_n^\epsilon$, and $\mu_n=\mu_n^\epsilon$.  Since $\mu_n>\mu_q$ by \eqref{eq:epsilon_characterized}, the sum of the $(n-q)$ smallest eigenvalues of the complex Hessian of $-\varphi$ is given by $\sum_{j=q+1}^n(-\mu_j)$, and this is strictly positive by \eqref{eq:epsilon_characterized}.  Hence, the complex Hessian of $-\varphi$ is $(n-q)$-positive on $\partial\Omega$.  Furthermore, \eqref{eq:n-q-1_positivity} implies that
\[
  \sum_{j=q+1}^n(-\mu_j)>-\mu_n=-\ddbar\varphi(L_n,\bar L_n),
\]
so we have condition (3) in the second set of hypotheses for Theorem \ref{thm:L^2 theory, with weights}.

As before, in Proposition~\ref{prop-main}, we take the bundle $E$ to be $T^{1,0}M$ (so $d=n$) and for the closed set $F\subset M$ we take $F=\partial\Omega$.  We take $g_0$ to be the extension of $h$ constructed above, and let $\varphi$ be the function $\phi$ of Proposition~\ref{prop-main}. Extend $g_0$ smoothly to a neighborhood of $\bd\Om$.  Then the hypotheses of Proposition~\ref{prop-main} are satisfied for $\tilde q=n-q$, and consequently there is a metric $g$ on a neighborhood of $\ol{\Om}$ which coincides with $g_0$ near $\partial\Omega$, and with respect to which $\mathcal{H}_{-\varphi}$ is strictly $(n-q)$-positive. Therefore, by Theorem~\ref{thm:L^2 theory, with weights}, we have
that $H^{p,q}_{L^2}(\Omega)=0$.

\subsection{Proof of Theorem~\ref{thm:q_equals_n_minus_one}}
The proof is the same as the proof of Theorem \ref{thm:continuously varying}, except that we replace the use of Theorem \ref{thm-keyconstruction} with that of Theorem \ref{thm-keyconstruction_top_degree}.

\section{Examples}

\subsection{Discontinuous Subbundles}
In Theorem \ref{thm:continuously varying}, we require the subbundle of shared positive directions to vary continuously on $\partial\Omega$, but it is possible for such a subbundle to exist pointwise without a continuous representative, as the following example illustrates.  This motivates the need for Theorem \ref{thm:q_equals_n_minus_one}, although at present we only have a proof of this result when $q=n-1$.

\begin{prop}
  For each $x\in\mathbb{R}^3$, define a Hermitian form on $\mathbb{C}^2\times\mathbb{C}^2$ by $H_0(u,v)=\left<u,v\right>$ and when $x\neq 0$
  \[
    H_x(u,v)=\begin{pmatrix}\bar v_1&\bar v_2\end{pmatrix}\begin{pmatrix}1-\exp(-|x|^{-2})(|x|-x_3)&-\exp(-|x|^{-2})(x_1+ix_2)\\-\exp(-|x|^{-2})(x_1-ix_2)&1-\exp(-|x|^{-2})(|x|+x_3)\end{pmatrix}\begin{pmatrix}u_1\\u_2\end{pmatrix}
  \]
  for all $u,v\in\mathbb{C}^2$.  Then:
  \begin{enumerate}
    \item For every $x\in\mathbb{R}^3$, there exists a non-trivial vector $v\in\mathbb{C}^2$ such that
    \begin{equation}
    \label{eq:positive}
      H_x(v,v)=|v|^2.
    \end{equation}
    \item If $\mathcal{O}\subset\mathbb{R}^3$ is an open set containing $\overline{B(0,R)}$ for some $R>0$ and $v:\mathcal{O}\rightarrow\mathbb{C}^2$ is a nowhere vanishing continuous vector field, then there exists $x\in\mathcal{O}$ at which
        \begin{equation}
        \label{eq:negative}
          H_x(v(x),v(x))=|v(x)|^2(1-2\exp(-R^{-2})R).
        \end{equation}
  \end{enumerate}
\end{prop}

\begin{proof}
  To prove \eqref{eq:positive}, given $x\in\mathbb{R}^3\backslash\{(0,0,-R):R\geq 0\}$, let $v=\begin{pmatrix}|x|+x_3\\-x_1+ix_2\end{pmatrix}$.  Then $v$ is an eigenvector of $H_x$ (with respect to the Euclidean metric) with eigenvalue $1$, so \eqref{eq:positive} follows.  When $x=(0,0,-R)$ for some $R>0$, \eqref{eq:positive} follows for $v=\begin{pmatrix}0\\1\end{pmatrix}$.  When $x=(0,0,0)$, \eqref{eq:positive} holds for any non-vanishing vector $v$.

  Now, let $\mathcal{O}\subset\mathbb{R}^3$ be an open set containing $\overline{B(0,R)}$ for some $R>0$ and let $v:\mathcal{O}\rightarrow\mathbb{C}^2$ be a nowhere vanishing continuous vector field.  Define a homeomorphism $\varphi:\mathbb{CP}^1\rightarrow S^2\subset\mathbb{R}^3$ by
  \[
    \varphi([z_1:z_2])=\left(\frac{2\re(z_1\bar z_2)}{|z_1|^2+|z_2|^2},\frac{2\im(z_1\bar z_2)}{|z_1|^2+|z_2|^2},\frac{|z_2|^2-|z_1|^2}{|z_1|^2+|z_2|^2}\right).
  \]
  One can check that the inverse $\varphi^{-1}:S^2\rightarrow\mathbb{CP}^1$ is given by
  \[
    \varphi^{-1}(x)=\begin{cases}\left[\frac{x_1+ix_2}{1+x_3}:1\right]&x\neq(0,0,-1)\\\left[1:0\right]&x=(0,0,-1)\end{cases}.
  \]
  For $t\in [0,R]$, define a continuous map $f_t:\mathbb{CP}^1\rightarrow\mathbb{CP}^1$ by
  \[
    f_t([z_1:z_2])=\left[v_1(t\varphi([z_1:z_2])):v_2(t\varphi([z_1:z_2]))\right].
  \]
  Since $f_t$ depends continuously on $t$, $f_t$ is a homotopy between $f_0$ and $f_R$.  Since $f_0$ is a constant, the degree of $f_0$ is zero, and hence the degree of $f_R$ is also zero.  This means that the induced map on the degree $2$ singular cohomology of $\mathbb{CP}^1\cong S^2$ is trivial.  Since the only other non-trivial singular cohomology group of $S^2$ is in degree $0$ and the induced map must be the identity on this group ($S^2$ has exactly one connected component), the Lefschetz Fixed-Point Theorem guarantees that $f_R$ has a fixed-point.

  Let $[z_1:z_2]\in\mathbb{CP}^1$ be a fixed-point to $f_R$.  Set $x=R\varphi([z_1:z_2])$, so that $\left[v_1(x):v_2(x)\right]=\varphi^{-1}\left(\frac{x}{R}\right)$.  If $x=(0,0,-R)$, we have $v_2(x)=0$, so
  \[
    H_x(v(x),v(x))=|v_1(x)|^2(1-\exp(-R^{-2})(R-(-R))),
  \]
  from which \eqref{eq:negative} follows.  If $x\neq (0,0,-R)$, then there exists $\lambda\in\mathbb{C}$ such that $v_1(x)=\lambda\frac{x_1+ix_2}{R+x_3}$ and $v_2(x)=\lambda$, so $v(x)=\frac{\lambda}{R+x_3}\begin{pmatrix}x_1+ix_2\\R+x_3\end{pmatrix}$.  Observe that when $|x|=R$, $\begin{pmatrix}x_1+ix_2\\R+x_3\end{pmatrix}$ is an eigenvector of $H_x$ with eigenvalue $(1-2\exp(-R^{-2})R)$.  Since $v(x)$ is a scalar multiple of this vector, $v(x)$ is also an eigenvector of $H_x$ with the same eigenvalue, and hence \eqref{eq:negative} follows.

\end{proof}
\begin{cor}
\label{cor:discontinuous_subbundle}
  There exists a convex, simply-connected domain $\mathcal{O}\subset\mathbb{R}^3$ and a family of Hermitian forms $H_x$ on $\mathbb{C}^2\times\mathbb{C}^2$ parameterized smoothly by $x\in\mathcal{O}$ such that:
  \begin{enumerate}
    \item For every $x\in\mathbb{R}^3$, there exists a vector $v\in\mathbb{C}^2$ such that $H_x(v,v)>0$.
    \item If $v:\mathcal{O}\rightarrow\mathbb{C}^2$ is a nowhere vanishing continuous vector field, then there exists $x\in\mathcal{O}$ at which $H_x(v(x),v(x))<0$.
  \end{enumerate}
\end{cor}

\begin{proof}
  It suffices to note that
  \[
    \lim_{R\rightarrow+\infty}1-2\exp(-R^{-2})R=-\infty,
  \]
  so for $R$ sufficiently large the right-hand side of \eqref{eq:negative} is strictly negative if we let $\mathcal{O}=B(0,2R)$.
\end{proof}

\subsection{Domains}
For $n\geq 1$ and $1\leq q\leq n$, set $\tilde M_q^n=\mathbb{C}^{n-q+1}\times\mathbb{CP}^{q-1}$.  If we equip $\mathbb{C}^{n-q+1}$ with holomorphic coordinates $\{z_j\}_1^{n-q+1}$, then $\varphi(z)=|z_1|^2+\cdots+|z_{n-q+1}|^2$ is an exhaustion function for $\tilde M_q^n$ such that the complex Hessian has precisely $n-q+1$ positive eigenvalues at every point of $\tilde M_q^n$.  Furthermore, the complex Hessian of $\varphi$ is strictly $q$-positive with respect to any Hermitian metric on $\tilde M_q^n$.  However, the complex Hessian of $\varphi$ has a kernel of dimension $2$ when $q\geq 3$, so we should not expect the second case in Theorem \ref{thm:continuously varying} to hold.  Suppose $q\geq 2$.  Any $C^2$ function on $\mathbb{CP}^{q-1}$ must have at least one point at which the complex Hessian has no positive eigenvalues, e.g., the point at which the function achieves its maximum on $\mathbb{CP}^{q-1}$.  This means that any $C^2$ function on $\tilde M_q^n$ must have at least one point at which the complex Hessian has at most $n-q+1$ positive eigenvalues, so it is impossible to build an exhaustion function for which the complex Hessian has at least $n-(q-2)=n-q+2$ positive eigenvalues at every point of $\tilde M_q^n$.

If $\Omega'\subset\mathbb{C}^{n-q+1}$ is bounded domain with smooth, strictly pseudoconvex boundary, then $\Omega=\Omega'\times\mathbb{CP}^{q-1}$ will be a domain in $\tilde M_q^n$ satisfying the hypotheses of Theorems \ref{thm:continuously varying} and \ref{thm-zq}.

For $n\geq 2$ and $1\leq q\leq n$, equip $\mathbb{CP}^n$ with homogeneous coordinates $[w_1:\ldots:w_{n+1}]$, i.e., we identify $\mathbb{CP}^n$ with $\mathbb{C}^{n+1}\backslash\{0\}$ under the equivalence relation $[w_1:\ldots:w_{n+1}]\sim[\lambda w_1:\ldots:\lambda w_{n+1}]$ for all $\lambda\in\mathbb{C}\backslash\{0\}$.  Set
\[
  |w|_+^2=\sum_{j=1}^{n-q+1}|w_j|^2\text{ and }|w|_-^2=\sum_{j=n-q+2}^{n+1}|w_j|^2,
\]
and define a non-compact submanifold $M_q^n\subset\mathbb{CP}^n$ by
\[
  M_q^n=\set{[w_1:\ldots:w_{n+1}]\in\mathbb{CP}^n:|w|_+^2<|w|_-^2}.
\]
Observe that the function
\[
  \phi(w)=-\log\left(1-\frac{|w|_+^2}{|w|_-^2}\right)
\]
induces a well-defined exhaustion function on $M_q^n$.  One can check that the complex Hessian of $\varphi$ has exactly $n-q+1$ positive eigenvalues.  Observe that $S=\{[w_1:\cdots:w_n]\in\mathbb{CP}^n:|w|_+=0\}$ is a $q-1$-dimensional complex submanifold of $M_q^n$, and hence, as argued above, any $C^2$ function on $M_q^n$ must have at least one point at which the complex Hessian has at most $n-q+1$ positive eigenvalues.  Indeed, the complex Hessian of $\phi$ has precisely $q-1$ negative eigenvalues on $M_q^n\backslash S$, but these $q-1$ eigenvalues vanish on $S$.

  Let $\{\mu_j\}_{1\leq j\leq n+1}$ be a sequence of real numbers such that $\mu_j\neq 0$ for any $1\leq j\leq n+1$, $\mu_j<0$ for at least one $1\leq j\leq n+1$, $\mu_j>1$ for all $1\leq j\leq n-q+1$, and $\mu_j>-1$ for all $n-q+2\leq j\leq n+1$.  Then $\sum_{j=1}^{n+1}\mu_j|w_j|^2> |w|_+^2-|w|_-^2$ for all $w\in\mathbb{C}^{n+1}$, so
  \[
    \Omega=\set{[w_1:\ldots:w_{n+1}]\in\mathbb{CP}^n:\sum_{j=1}^{n+1}\mu_j|w_j|^2<0}
  \]
  is a subset of $M_q^n$ satisfying the hypotheses of Theorems \ref{thm-zq} and \ref{thm:continuously varying}..

 \bibliographystyle{alpha}
 \bibliography{mybib,inertia}
\end{document}